\documentclass{amsart}

\usepackage{amsopn}
\usepackage{amssymb,amscd,amssymb,amsfonts}
\usepackage{amsthm}
\usepackage{anysize}
\usepackage{array}
\usepackage{bbding}
\usepackage{bbm}
\usepackage{booktabs} 
\usepackage{calc}
\usepackage{CJK,fancyhdr,amscd}
\usepackage{color}
\usepackage{dynkin-diagrams}
\usepackage{enumerate}
\usepackage{enumitem}
\usepackage{epsfig,enumerate}
\usepackage{faktor}
\usepackage{float}
\usepackage{geometry}
\usepackage{graphicx, graphics, epsfig}

\usepackage{hyperref}

\usepackage{indentfirst, latexsym}
\usepackage{latexsym}
\usepackage{lie-hasse}
\usepackage[mathcal]{euscript}
\usepackage{mathrsfs,dsfont}
\usepackage{multirow}
\usepackage{pgfplots}
\usepackage{setspace}

\usepackage{tabto}
\usepackage{threeparttable}
\usepackage{tikz}
\usepackage{tikz-cd}
\usepackage{wrapfig}
\usepackage[all,poly,knot]{xy}

\usetikzlibrary{babel}
\usetikzlibrary{positioning}
\newcommand\restr[2]{
  \left.\kern-\nulldelimiterspace
  #1
  \vphantom{\big|}
  \right|_{#2}
}

\newcommand{\nc}{\newcommand}

\nc{\ag}{\mathfrak{a}}
\nc{\bg}{\mathfrak{b}}
\nc{\cg}{\mathfrak{c}}
\nc{\dg}{\mathfrak{d}} 
\nc{\eg}{\mathfrak{e}}
\nc{\fg}{\mathfrak{f}}
\nc{\ggo}{\mathfrak{g}} 
\nc{\hg}{\mathfrak{h}}
\nc{\ig}{\mathfrak{i}}
\nc{\jg}{\mathfrak{j}}
\nc{\kg}{\mathfrak{k}}
\nc{\lgo}{\mathfrak{l}}
\nc{\mg}{\mathfrak{m}}
\nc{\ngo}{\mathfrak{n}}
\nc{\og}{\mathfrak{o}}
\nc{\pg}{\mathfrak{p}}
\nc{\qg}{\mathfrak{q}}
\nc{\rg}{\mathfrak{r}}
\nc{\sg}{\mathfrak{s}}
\nc{\tg}{\mathfrak{t}}
\nc{\ug}{\mathfrak{u}}
\nc{\vg}{\mathfrak{v}}
\nc{\wg}{\mathfrak{w}}
\nc{\xg}{\mathfrak{x}}
\nc{\yg}{\mathfrak{y}}
\nc{\zg}{\mathfrak{z}}

\nc{\afr}{\mathfrak{a}}
\nc{\bfr}{\mathfrak{b}}
\nc{\cfr}{\mathfrak{c}}
\nc{\dfr}{\mathfrak{d}} 
\nc{\efr}{\mathfrak{e}}
\nc{\ffr}{\mathfrak{f}}
\nc{\gfr}{\mathfrak{g}} 
\nc{\hfr}{\mathfrak{h}}
\nc{\ifr}{\mathfrak{i}}
\nc{\jfr}{\mathfrak{j}}
\nc{\kfr}{\mathfrak{k}}
\nc{\lfr}{\mathfrak{l}}
\nc{\mfr}{\mathfrak{m}}
\nc{\nfr}{\mathfrak{n}}
\nc{\ofr}{\mathfrak{o}}
\nc{\pfr}{\mathfrak{p}}
\nc{\qfr}{\mathfrak{q}}
\nc{\rfr}{\mathfrak{r}}
\nc{\sfr}{\mathfrak{s}}
\nc{\tfr}{\mathfrak{t}}
\nc{\ufr}{\mathfrak{u}}
\nc{\vfr}{\mathfrak{v}}
\nc{\wfr}{\mathfrak{w}}
\nc{\xfr}{\mathfrak{x}}
\nc{\yfr}{\mathfrak{y}}
\nc{\zfr}{\mathfrak{z}}

\nc{\Afr}{\mathfrak{A}}
\nc{\Bfr}{\mathfrak{B}}
\nc{\Cfr}{\mathfrak{C}}
\nc{\Dfr}{\mathfrak{D}} 
\nc{\Efr}{\mathfrak{E}}
\nc{\Ffr}{\mathfrak{F}}
\nc{\Gfr}{\mathfrak{G}} 
\nc{\Hfr}{\mathfrak{H}}
\nc{\Ifr}{\mathfrak{I}}
\nc{\Jfr}{\mathfrak{J}}
\nc{\Kfr}{\mathfrak{K}}
\nc{\Lfr}{\mathfrak{L}}
\nc{\Mfr}{\mathfrak{M}}
\nc{\Nfr}{\mathfrak{N}}
\nc{\Ofr}{\mathfrak{O}}
\nc{\Pfr}{\mathfrak{P}}
\nc{\Qfr}{\mathfrak{Q}}
\nc{\Rfr}{\mathfrak{R}}
\nc{\Sfr}{\mathfrak{S}}
\nc{\Tfr}{\mathfrak{T}}
\nc{\Ufr}{\mathfrak{U}}
\nc{\Vfr}{\mathfrak{V}}
\nc{\Wfr}{\mathfrak{W}}
\nc{\Xfr}{\mathfrak{X}}
\nc{\Yfr}{\mathfrak{Y}}
\nc{\Zfr}{\mathfrak{Z}}

\nc{\ggob}{\overline{\mathfrak{g}}} \nc{\sog}{\mathfrak{so}}
\nc{\sug}{\mathfrak{su}}
\nc{\spg}{\mathfrak{sp}}
\nc{\slg}{\mathfrak{sl}}
\nc{\glg}{\mathfrak{gl}}  
\nc{\affg}{\mathfrak{aff}} 

\nc{\aca}{\mathcal{A}}
\nc{\bca}{\mathcal{B}}
\nc{\cca}{\mathcal{C}}
\nc{\dca}{\mathcal{D}}
\nc{\eca}{\mathcal{E}}
\nc{\fca}{\mathcal{F}}
\nc{\gca}{\mathcal{G}}
\nc{\hca}{\mathcal{H}}
\nc{\ica}{\mathcal{I}}
\nc{\jca}{\mathcal{J}}
\nc{\kca}{\mathcal{K}}
\nc{\lca}{\mathcal{L}}
\nc{\mca}{\mathcal{M}}
\nc{\nca}{\mathcal{N}}
\nc{\oca}{\mathcal{O}}
\nc{\pca}{\mathcal{P}}
\nc{\qca}{\mathcal{Q}}
\nc{\rca}{\mathcal{R}}
\nc{\sca}{\mathcal{S}}
\nc{\tca}{\mathcal{T}}
\nc{\uca}{\mathcal{U}}
\nc{\vca}{\mathcal{V}}
\nc{\wca}{\mathcal{W}}
\nc{\xca}{\mathcal{X}}
\nc{\yca}{\mathcal{Y}}
\nc{\zca}{\mathcal{Z}}

\nc{\Amb}{\mathbb{A}}
\nc{\Bmb}{\mathbb{B}}
\nc{\Cmb}{\mathbb{C}}
\nc{\Dmb}{\mathbb{D}}
\nc{\Emb}{\mathbb{E}}
\nc{\Fmb}{\mathbb{F}}
\nc{\Gmb}{\mathbb{G}}
\nc{\Hmb}{\mathbb{H}}
\nc{\Imb}{\mathbb{I}}
\nc{\Jmb}{\mathbb{J}}
\nc{\Kmb}{\mathbb{K}}
\nc{\Lmb}{\mathbb{L}}
\nc{\Mmb}{\mathbb{M}}
\nc{\Nmb}{\mathbb{N}}
\nc{\Omb}{\mathbb{O}}
\nc{\Pmb}{\mathbb{P}}
\nc{\Qmb}{\mathbb{Q}}
\nc{\Rmb}{\mathbb{R}}
\nc{\Smb}{\mathbb{S}}
\nc{\Tmb}{\mathbb{T}}
\nc{\Umb}{\mathbb{U}}
\nc{\Vmb}{\mathbb{V}}
\nc{\Wmb}{\mathbb{W}}
\nc{\Xmb}{\mathbb{X}}
\nc{\Ymb}{\mathbb{Y}}
\nc{\Zmb}{\mathbb{Z}}

\nc{\vp}{\varphi}
\nc{\ddt}{\tfrac{d}{dt}}
\nc{\dsdt}{\tfrac{d^2}{dt^2}}
\nc{\dds}{\frac{d}{ds}}
\nc{\dpar}{\frac{\partial}{\partial t}} \nc{\im}{\mathrm{i}}

\nc{\SO}{\mathrm{SO}}
\nc{\Spe}{\mathrm{Sp}}
\nc{\Sl}{\mathrm{SL}}
\nc{\SU}{\mathrm{SU}}
\nc{\Or}{\mathrm{O}}
\nc{\U}{\mathrm{U}}
\nc{\Gl}{\mathrm{GL}}
\nc{\Se}{\mathrm{S}}
\nc{\Cl}{\mathrm{Cl}}
\nc{\Spin}{\mathrm{Spin}}
\nc{\Pin}{\mathrm{Pin}}
\nc{\G}{\mathrm{GL}_n(\RR)} \nc{\g}{\mathfrak{gl}_n(\RR)}
\nc{\Eg}{\mathrm{E}}
\nc{\Fg}{\mathrm{F}}
\nc{\Gg}{\mathrm{G}}

\nc{\RR}{{\Bbb R}}
\nc{\HH}{{\Bbb H}}
\nc{\CC}{{\Bbb C}}
\nc{\ZZ}{{\Bbb Z}}
\nc{\FF}{{\Bbb F}}
\nc{\NN}{{\Bbb N}}
\nc{\QQ}{{\Bbb Q}}
\nc{\PP}{{\Bbb P}}
\nc{\OO}{{\Bbb O}}

\nc{\vs}{\vspace{.2cm}}
\nc{\vsp}{\vspace{1cm}} \nc{\ip}{\langle\cdot,\cdot\rangle}
\nc{\ipp}{(\cdot,\cdot)}
\nc{\la}{\langle}
\nc{\ra}{\rangle}
\nc{\unm}{\tfrac{1}{2}}
\nc{\unc}{\tfrac{1}{4}}
\nc{\und}{\frac{1}{16}}
\nc{\no}{\vs\noindent}
\nc{\lam}{\rho^2(\RR^n)^*\otimes\RR^n}
\nc{\tangz}{{\rm T}^{\rm Zar}}
\nc{\nor}{{\sf n}}
\nc{\mum}{/\!\!/}
\nc{\kir}{/\!\!/\!\!/}
\nc{\Ri}{\tfrac{4\Ric_{\mu}}{||\mu||^2}} \nc{\ds}{\displaystyle}
\nc{\ben}{\begin{enumerate}}
\nc{\een}{\end{enumerate}}
\nc{\benalp}{\begin{enumerate}[font = \normalfont, label = (\alph*)]}
\nc{\f}{\frac}
\nc{\lb}{[\cdot,\cdot]}
\nc{\isn}{\tfrac{1}{||v||^2}}
\nc{\gkp}{(\ggo=\kg\oplus\pg,\ip)} \nc{\ukh}{(\ug=\kg\oplus\hg,\ip)}
\nc{\tgkp}{(\tilde{\ggo}=\kg\oplus\pg,\ip)}
\nc{\wt}{\widetilde}
\nc{\iop}{\mathtt{i}}
\nc{\jop}{\mathtt{j}} 
\nc{\Hk}{H_{\kil}}
\nc{\gk}{g_{\kil}}

\nc{\abel}{\operatorname{ab}}
\nc{\ad}{\operatorname{ad}}
\nc{\Ad}{\operatorname{Ad}}
\nc{\Adj}{\operatorname{Adj}}
\nc{\Aff}{\operatorname{Aff}}
\nc{\Aut}{\operatorname{Aut}}
\nc{\cas}{\operatorname{C}}
\nc{\CH}{\operatorname{CH}}
\nc{\CCone}{\operatorname{CC}}
\nc{\Cone}{{\mathcal C}}
\nc{\CP}{{\mathcal P}}
\nc{\Cric}{\operatorname{P}}
\nc{\CRic}{\operatorname{PP}}
\nc{\Crit}{\operatorname{Crit}}
\nc{\Der}{\operatorname{Der}}
\nc{\dete}{\operatorname{det}}
\nc{\Diag}{\operatorname{Dg}} \nc{\Diagg}{\operatorname{Diag}}
\nc{\dif}{\operatorname{d}}
\nc{\Diff}{\operatorname{Diff}}
\nc{\E}{\operatorname{E}}
\nc{\End}{\operatorname{End}}
\nc{\grad}{\operatorname{grad}}
\nc{\herm}{\operatorname{herm}}
\nc{\Hess}{\operatorname{Hess}}
\nc{\id}{\operatorname{id}}
\nc{\Ima}{\operatorname{Im}}
\nc{\inj}{\operatorname{inj}}
\nc{\Inn}{\operatorname{Inn}}
\nc{\Irr}{\operatorname{Irr}}
\nc{\Iso}{\operatorname{Iso}}
\nc{\Ker}{\operatorname{Ker}}
\nc{\kil}{\operatorname{B}}
\nc{\Le}{\operatorname{L}}
\nc{\level}{\operatorname{level}}
\nc{\lic}{\operatorname{L}}
\nc{\Lie}{\operatorname{L}}
\nc{\mcc}{\operatorname{mcc}} 
\nc{\mm}{\operatorname{m}}
\nc{\Mm}{\operatorname{M}}
\nc{\modu}{\operatorname{mod}}
\nc{\Nor}{\operatorname{Norm}}
\nc{\Order}{\operatorname{O}}
\nc{\proy}{\operatorname{pr}}
\nc{\rad}{\operatorname{r}}
\nc{\rank}{\operatorname{rk}}
\nc{\Rea}{\operatorname{Re}}
\nc{\ricac}{\operatorname{Rc^{ac}}}
\nc{\Ricac}{\operatorname{Ric^{ac}}}
\nc{\ricc}{\operatorname{Rc^{c}}}
\nc{\Ricc}{\operatorname{Ric^{c}}}
\nc{\ricci}{\operatorname{Rc}}
\nc{\Ricci}{\operatorname{Ric}}
\nc{\riccig}{\operatorname{ric^{\gamma}}}
\nc{\Riem}{\operatorname{Rm}}
\nc{\scalar}{\operatorname{Sc}}
\nc{\Sec}{\operatorname{Sec}}
\nc{\sen}{\operatorname{sen}}
\nc{\spann}{\operatorname{span}}
\nc{\Spec}{\operatorname{Spec}}
\nc{\sym}{\operatorname{sym}}
\nc{\symac}{\operatorname{sym^{ac}}}
\nc{\symc}{\operatorname{sym^{c}}}
\nc{\tang}{\operatorname{T}}
\nc{\tr}{\operatorname{tr}}
\nc{\val}{\operatorname{val}}
\nc{\vol}{\operatorname{vol}} 

\nc{\DynkinF}{	\begin{figure}[h]
		\centering
		\begin{tikzpicture}[scale = 1, baseline={(0,-0.1)}]
			\tikzstyle{every node}=[circle, draw, inner sep=1.5pt];
			
			\node (a3) at (0,0) {};
			\node (a4) at (1,0) {};
			\node (a5) at (2,0) {};
			\node (a6) at (3,0) {};
			
			\node[draw = none] at ($(a3.south)+(0,-8pt)$) {$\al_1$};
			\node[draw = none] at ($(a4.south)+(0,-8pt)$) {$\al_2$};
			\node[draw = none] at ($(a5.south)+(0,-8pt)$) {$\al_3$};
			\node[draw = none] at ($(a6.south)+(0,-8pt)$) {$\al_4$};
			
			\draw (a3)--(a4) (a5)--(a6);
			\draw[double distance=2pt] (a4) -- (a5);
			\draw[double distance=2pt,-implies] (a4)-- ($(a5)!0.4!(a4)$);
		\end{tikzpicture}
\end{figure}}

\newcommand{\DynF}[4]{
	\begin{tikzpicture}[scale = 1, baseline={(0,-0.1)}]
		\tikzstyle{dynkinNode}=[circle, draw, inner sep=1.5pt];
		
		\node[dynkinNode, fill=#1] (a1) at (0,0) {};

		\node[dynkinNode, fill=#2] (a2) at (1,0) {};

		\node[dynkinNode, fill=#3] (a3) at (2,0) {};

		\node[dynkinNode, fill=#4] (a4) at (3,0) {};
		
		\node[draw = none] at ($(a1.south)+(0,-8pt)$) {\small $\al_1$};
		\node[draw = none] at ($(a2.south)+(0,-8pt)$) {\small $\al_2$};
		\node[draw = none] at ($(a3.south)+(0,-8pt)$) {\small $\al_3$};
		\node[draw = none] at ($(a4.south)+(0,-8pt)$) {\small $\al_4$};
		
		\draw (a1)--(a2) (a3)--(a4);
		\draw[double distance=2pt] (a2) -- (a3);
		\draw[double distance=2pt,-implies] (a2)-- ($(a3)!0.4!(a2)$);
	\end{tikzpicture}
}

\nc{\DynG}[2]{\begin{tikzpicture}[scale = 1, baseline={(0,-0.1)}]
		\tikzstyle{dynkinNode}=[circle, draw, inner sep=1.5pt];
		
		\node[dynkinNode, fill=#1] (a1) at (0,0) {};

		\node[dynkinNode, fill=#2] (a2) at (1,0) {};
		
		\node[draw = none] at ($(a1.south)+(0,-8pt)$) {\small $\al_1$};
		\node[draw = none] at ($(a2.south)+(0,-8pt)$) {\small $\al_2$};
			
		\draw[double distance=2pt] (a1) -- (a2);
		\draw[double distance=2pt, -implies] (a1)--($(a2)!0.4!(a1)$);
		\draw (a1)--(a2);	
		\end{tikzpicture}
}

\nc{\col}[1]{\mathcal{C}_{#1}}
\nc{\tri}[1]{\mathcal{T}_{#1}}

\nc{\al}{\alpha}

\nc{\seq}{\subseteq}
\nc{\sneq}{\subsetneq}

\renewcommand{\arraystretch}{1.3}

\theoremstyle{plain}
\newtheorem{theorem}{Theorem}[section]

\newtheorem{lemma}[theorem]{Lemma}

\theoremstyle{definition}
\newtheorem{definition}[theorem]{Definition}

\newtheorem{conjecture}[theorem]{Conjecture}

\theoremstyle{remark}
\newtheorem{remark}[theorem]{Remark}

\title[Parallel Bismut torsion and isotropy representation]{The parallel Bismut torsion condition and the isotropy representation of flag manifolds}

\author{Martiniano Faure}
\email{martiniano.faure@mi.unc.edu.ar}  

\address{FAMAF, Universidad Nacional de C\'ordoba and CIEM-CONICET, Av. Medina Allende s/n, Ciudad Universitaria, X5000HUA C\'ordoba (Argentina)}

\thanks{We would like to acknowledge support from the ICTP (Italy) through the Associates Programme and from the Simons Foundation through grant number 284558FY19, and support from the Consejo Interuniversitario Nacional (Argentina) through an undergraduate research fellowship.}

\date{\today}

\begin{document}

\maketitle

\begin{abstract}
	A classification of all invariant Hermitian structures with parallel Bismut torsion on most flag manifolds is given.  More precisely, we prove that on a flag manifold $M=G/H$ such that $G$ is simple and different from $\Eg_6, \Eg_7, \Eg_8$ if $H$ is non-abelian, a $G$-invariant Hermitian structure on $M$ has parallel Bismut torsion if and only if the metric is either K\"ahler or a multiple of the Killing or standard metric. 
\end{abstract}

\section{Introduction}\label{Sec1}

On any Hermitian manifold, there exists a unique Hermitian connection $\nabla^B$ with skew-symmetric torsion called the {\it Bismut connection} (or {\it Strominger}).  The geometry of Hermitian manifolds with {\it parallel Bismut torsion}, i.e., $\nabla^BT^B=0$ ({\it BTP} for short) has attracted considerable attention in the last years (see \cite{ZZ1, ZZ2} and the references therein).    

The study of the BTP condition in the compact homogeneous case was initiated in \cite{PodZ}, where the following results were obtained:
\begin{enumerate}[font = \normalfont, label = (\alph*)]
	\item \cite[Theorem 1.3]{PodZ} Any $G$-invariant BTP metric on a flag manifold $G/H$ with $G$ a compact simple Lie group and two isotropy summands is either K\"ahler or {\it standard} (i.e., the restriction of $-\kil_\ggo$, where $\kil_\ggo$ is the Killing form of $\ggo$) up to scaling.  Furthermore, the same holds for the full flag $M=\SU(n+1)/T^n$ (see \cite[Theorem 1.4]{PodZ}).  
	
	\item \cite[Theorem 1.5]{PodZ} On a given compact simple Lie group $G$ with a maximal torus $T$, a left and $\Ad(T)$-invariant metric $g$ is BTP if and only if up to scaling, $g$ coincides with the restriction of $-\kil_\ggo$ on the orthogonal complement of $\tg$.  
\end{enumerate}

It is conjectured in \cite[Conjecture 1.6]{PodZ} that part (a) holds for any flag manifold. In this paper, our main objective is to confirm the validity of the previous conjecture for most flag manifolds.   

\begin{theorem}\label{thm}
	Let $M=G/H$ be a flag manifold such that $G$ is simple and different from $\Eg_6, \Eg_7, \Eg_8$ if $H$ is non-abelian.  Then a $G$-invariant Hermitian structure on $M$ is BTP if and only if the metric is either K\"ahler or a multiple of the standard metric. 
\end{theorem}

A {\it flag manifold} is a homogeneous space $F=G/H$, where $G$ is a compact and connected semisimple Lie group and $H$ is the centralizer in $G$ of some torus of $G$.  As is well known, flag manifolds are the only compact manifolds with finite fundamental group admitting a homogeneous symplectic structure, or a homogeneous K\"ahler metric. If for a maximal torus $T\seq H\seq G$ we consider the corresponding root system $\Delta \seq (\tg^\Cmb)^*$ of $\ggo^\Cmb:=\ggo\otimes\CC$, then the $\kil_\ggo$-orthogonal reductive decomposition $\ggo=\hg\oplus\qg$ of $F=G/H$, $T_oF\equiv \qg$, can be further decomposed as  
\begin{equation}\label{hq}
	\hg^\Cmb=\tg^\Cmb\oplus\bigoplus_{\alpha\in\Delta_\hg}\ggo_\alpha, \qquad 
	\qg^\Cmb=\bigoplus_{\alpha\in\Delta_\qg}\ggo_\alpha, \qquad \Delta=\Delta_\hg\sqcup\Delta_\qg,
\end{equation}
where $\Delta_\hg$ is the root system of the semisimple Lie algebra $[\hg,\hg]$ and $\Delta_\qg$ is the set of {\it complementary roots}.  The decomposition $\qg=\qg_1\oplus\dots\oplus\qg_r$ of the isotropy representation in irreducible components is determined by the coincidence of the restrictions of the roots to the center of $\hg$.  

After some preliminaries in Section \ref{Sec2}, we compute in Section \ref{Sec3} the number $r$ of irreducible components of the isotropy representation of any flag manifold whenever $G$ is different from $\Eg_6$, $\Eg_7$ and $\Eg_8$. These formulas for $r$, which are given in Tables \ref{tab:classical} and \ref{tab:expectF4andG2}, may be of independent interest in the study of other geometric properties of flag manifolds. This problem was already addressed by Alekseevsky in \cite{Alek}, for a large family of flag manifolds. 

The proof of Theorem \ref{thm} is worked out in Section \ref{Sec4} \footnote{It is worth mentioning that this section continues the research developed in the author's master's thesis (Universidad Nacional de Córdoba, Córdoba, Argentina; 2026)}. We first show that any BTP metric $g=(x_\alpha)_{\alpha\in\Delta_\qg^+}$ (i.e., $g|_{\qg_\alpha}=x_\alpha(-\kil_\ggo)|_{\qg_\alpha}$, $x_\alpha>0$, where $\qg_\alpha^\Cmb=\ggo_\alpha\oplus\ggo_{-\alpha}$) satisfies the following combinatorial property: 
$$
\forall \alpha,\beta,\alpha+\beta\in\Delta_\qg^+, \quad\mbox{either}\quad 
x_{\alpha+\beta}=x_\alpha+x_\beta \quad\mbox{or}\quad x_{\alpha+\beta}=x_\alpha=x_\beta.  
$$
The rest of the proof consists in showing that this local property is necessarily global, via a case-by-case analysis among the different types of simple Lie algebras.

\section{Preliminaries}\label{Sec2}

\subsection{Root systems}\label{Sec2,1}

Let $(M, J)$ be a complex manifold. We say that $(M,J)$ is \textit{K\"ahlerian} if there exists some Riemannian metric $g$ on $M$ such that $(M, g, J)$ is a K\"ahler manifold. The object of study in this article will be $G$-homogeneous compact K\"ahlerian manifolds $(M, J)$, where $G$ is a Lie group, $J$ is a $G$-invariant complex structure and $M$ has finite fundamental group. If such a manifold admits an invariant K\"ahler metric, then it is called a \textit{flag manifold}.

\begin{theorem}\cite{FlagKähler}
	Suppose that $(M, J)$ is a $G$-invariant complex manifold. Then $(M, J)$ is a flag manifold if and only if $M = G/H$, where $G$ is a compact semisimple Lie group and $H = C_G(T')$ is the centralizer of a torus $T' < G$.
\end{theorem}

\begin{remark}
	In general, the torus $T'$ may not be maximal. Hence $H$ may not be abelian. In case that $T'$ is a maximal torus, then $T' = H$ and the quotient $M = G/T'$ is called a \textit{full flag manifold}.
\end{remark}

From here onwards we will always assume $G$ to be a compact semisimple Lie group and $H$ to be the centralizer of some torus $T' < G$. The homogeneous space $G/H$ is then reductive. Let $T$ be a maximal torus of $G$ such that $T' \seq T \seq H$. The inclusion of the real Lie algebras $\tfr \seq \hfr \seq \gfr$ induces \textit{root spaces} decompositions for the complex Lie algebras $\gfr^\Cmb$ and $\hfr^\Cmb$,
$$\gfr^\Cmb = \tfr^\Cmb \oplus \bigoplus_{\al \in \Delta} \gfr_{\al}, \hspace{5mm} \hfr^\Cmb = \tfr^\Cmb \oplus \bigoplus_{\al \in \Delta_{\hfr}} \gfr_{\al},$$
where $\Delta \seq (\tfr^\Cmb)^* - \{ 0\}$ is the root system of $\gfr$ and $\Delta_\hfr := \{ \al \in \Delta \colon \al\rvert_{\zfr(\hfr)} \equiv 0 \}$ is the root system of $[\hfr, \hfr]$. The root system is then partitioned as $\Delta = \Delta_{\hfr} \sqcup \Delta_{\qfr}$, where $\Delta_{\qfr}$ is the set of \textit{complementary roots}.

For any $\al \in \Delta$, the \textit{root space} $\gfr_{\al} = \{ E \in \gfr^{\Cmb} \colon [h, E] = \al(h) E, \hspace{1mm} \forall h \in \tfr^\Cmb \}$ is a common eigenspace for the commuting family of lineal operators $\ad(\tfr^\Cmb)$ on $\gfr^\Cmb$. These complex subspaces are 1-dimensional and they satisfy that
$$
[\gfr_\al,\gfr_\beta]
\left\{\begin{array}{ll}
	= \gfr_{\al+\beta}, &\qquad \al+\beta\in\Delta, \\
	\seq \tfr^{\Cmb}, &\qquad \al+\beta = 0, \\
	= 0, &\qquad \text{otherwise}. 
\end{array}\right.
$$
An \textit{ordering} for $\Delta$ is a partition $\Delta = \Delta^+ \sqcup \Delta^{-}$ such that $(\Delta^+ + \Delta^+)\cap \Delta \seq \Delta^+$, where $\Delta^- := - \Delta^+$. Given an ordering $\Delta^+$ for $\Delta$, we consider a \textit{basis} $\Pi \seq \Delta^+$ for $\Delta$, i.e., a basis of $(\tfr^\Cmb)^*$ such that every $\beta \in \Delta^+$ may be written as a $\Nmb_0$-linear combination of the elements in $\Pi$. We denote $\Delta_{\qfr}^+ := \Delta_{\qfr} \cap \Delta^+$ and $\Delta_{\hfr}^+ := \Delta_{\hfr} \cap \Delta^+$.

Since the Killing form $\kil_{\gfr^{\Cmb}}$ is non-degenerate on $\tfr^{\Cmb} \times \tfr^{\Cmb}$, we obtain an identification $\tfr^{\Cmb} \equiv (\tfr^{\Cmb})^*$. Therefore, for each $\al \in (\tfr^\Cmb)^*$, there exists a unique vector $H_\al \in \tfr^\Cmb$ such that $\al = \kil_{\gfr^{\Cmb}}(H_\al, \cdot)$. There also exist elements $E_\al \in \gfr_\al$ such that
$$\kil_{\gfr^{\Cmb}}(E_\al, E_{-\al})=1, \quad [E_\al, E_{-\al}] = H_\al.$$
Then for any basis $\Pi \seq \Delta$ for the root system, the set $\{ H_\al \colon \al \in \Pi\} \cup \{ E_\al \colon \al \in \Delta\}$ is a basis for $\gfr^{\Cmb}$. We denote the structure constants in this basis by $N_{\al, \beta}$, which are given by 
$$[E_\al, E_\beta] = N_{\al, \beta} E_{\al, \beta}$$
for any $\al, \beta \in \Delta$ such that $\al+\beta \neq 0$. Note that $N_{\al, \beta} = 0$ if $\al+\beta \notin \Delta$ and $\alpha+\beta \neq 0$.

Given $\al \in \Delta$, we define $\qfr_\al := (\gfr_{\al} \oplus \gfr_{-\al}) \cap \gfr$, which is a real form for $\ggo_{\al}\oplus \ggo_{-\al}$, i.e., $\qfr_\al^{\Cmb} = \gfr_{\al}\oplus \gfr_{-\al}$. The set $\{ e_{\al}, f_{\al}\}$ is a basis for $\qfr_\al$, where $e_{\al} := \tfrac{1}{\sqrt{2}}(E_{\al} - E_{-\al})$ and $f_\al := \tfrac{i}{\sqrt{2}}(E_{\al} + E_{-\al})$. Also $\gfr = \hfr \oplus \qfr$ is a reductive decomposition for $\gfr$, where $\qfr := \bigoplus_{\al \in \Delta_\qfr^+} \qfr_\al$.

Consider the following relation on $\Delta_{\qfr}$
$$\al \sim \beta \text{ if and only if } \al\rvert_{\zfr(\hfr)} = \beta\rvert_{\zfr(\hfr)}.$$
It is easy to check that $\sim$ is an equivalence relation. For each $\beta \in \Delta_{\qfr}^+$, we denote its equivalence class by $P_\beta$. Equivalently, given the corresponding basis $\Pi \seq \Delta^+$, define $\Pi_{\hfr} := \Pi \cap \Delta_{\hfr}$ and $\Pi_{\qfr} := \Pi \cap \Delta_{\qfr}$. Observe that $\Pi_{\hfr}$ is a basis for $\Delta_\hfr$. Every root $\beta \in \Delta_{\qfr}^+$ may be written as a lineal combination of the form
$$\beta = \sum_{\al \in \Pi_{\hfr}} m_{\beta, \al} \al + \sum_{\al \in \Pi_{\qfr}} n_{\beta, \al}\al,$$
where $m_{\beta, \al}$, $n_{\beta, \al} \in \Nmb_0$. Then the equivalence class of $\beta$ can be described as
$$P_\beta = \{ \gamma \in \Delta_{\qfr}^+ \colon n_{\gamma, \al} = n_{\beta, \al}, \hspace{1mm} \forall \al \in \Pi_\qfr\}.$$

Let $P_1, \dots, P_r$ be the distinct equivalence classes for the relation $\sim$ in $\Delta_\qfr^+$. For $1 \leq i \leq r$, define $\qfr_i := \bigoplus_{\al \in P_i} \qfr_\al$. Then $\qfr = \qfr_1 \oplus \dots \oplus \qfr_r$ is a decomposition of $\qfr$ in irreducible $\Ad(H)$-invariant subspaces, which are pairwise inequivalent as $\Ad(H)$-representations. We will denote the amount of such subspaces by $r := r(G/H)$.

\subsection{Homogeneous Hermitian geometry}\label{Sec2,2}
We now turn to the study of Riemannian metrics and almost complex structures on the homogeneous compact manifold $G/H$. Let $o := eH$ be the \textit{origin} in $G/H$. Note that the smooth left action $G \curvearrowright G/H$ gives rise to a $\Rmb$-lineal map $\hat{\theta} : \gfr \rightarrow T_o(G/H)$, $\hat{\theta}(X) := X_o^*$, where $X^* \in \Xfr(M)$ is defined by
$$X_{aH}^* := \tfrac{d}{dt}\rvert_{t = 0} (exp(tX) \cdot aH)$$
for all $a \in G$. Observe that $X_o^* = 0$ if and only if $X \in \hfr$. Therefore $\hat{\theta}$ induces an identification $\qfr \equiv T_o(G/H)$. From this identification, a description of almost complex structures on $G/H$ can be obtained.

\begin{theorem}\label{CompInvRed}\cite[Chapter X, Proposition 6.5]{Nom2}
	The $G$-invariant almost complex structures on $G/H$ are in one to one correspondance with the $\Ad(H)$-invariant almost complex structures on $\qfr$.
\end{theorem}

For flag manifolds, any $G$-invariant almost complex structure $J$ is of the form $Je_\al = \epsilon_{\al}f_\al$, $Jf_{\al} = -\epsilon_{\al}e_{\al}$, for some coefficients $\epsilon_{\al}$ satisfying $\epsilon_{\al} = \pm 1$, $\epsilon_{-\al} = -\epsilon_{\al}$ and $\epsilon_{\al} = \epsilon_{\beta}$ if $\al \sim \beta$, for any $\al, \beta \in \Delta_{\qfr}$. Additionally, $J$ is integrable if and only if the set $P(J) := \{ \al \in \Delta_\qfr \colon \epsilon_{\al} = 1 \}$ satisfies that $P(J) = \Delta_\qfr \cap \Delta^+$ for some ordering $\Delta^+$ for $\Delta$ (see \cite[Lemma 12.4]{BorHir}). Then on every flag manifold, there exist only finite $G$-invariant (almost) complex structures, which are in one to one correspondence with the possible orderings $\Delta^+$ for $\Delta$. In general however, they are not necessarily pairwise biholomorphic. 

\begin{theorem}
	Given two $G$-invariant complex structures $J_1$, $J_2$ on a flag manifold $G/H$, the following statements are equivalent:
	\begin{enumerate}[font = \normalfont, label = (\alph*)]
			 \item There exists a biholomorphism between the complex manifolds $(G/H,J_1)$, $(G/H,J_2)$.
			 \item There exists some $\psi \in \mathrm{Aut}(\Delta)$ such that $\psi(\Delta_\hfr) = \Delta_\hfr$ and $\psi( P_1) = P_2$, where $P_i := P(J_i)$ \cite[Theorem 13.8]{BorHir}.
			 \item There exists some $\psi \in \mathrm{Aut}(\Delta)$ such that $\psi(\Pi_1) = \Pi_2$ and $\psi(\Pi_\hfr) = \Pi_\hfr$, where $\Pi_i$ is the basis for the ordering $\Delta^+_i \seq \Delta$ such that $\Delta^+_i \cap \Delta_\qfr = P(J_i)$ \cite[Theorems 2, 4]{Nishi}.
		\end{enumerate}
\end{theorem}

In this article, a \textit{complex flag manifold} will always be a flag manifold $G/H$ endowed with a $G$-invariant complex structure $J$. The complex flag manifold $(G/H, J)$ will be said to be a \textit{coomplex full flag manifold} if $G/H$ is a full flag manifold.

We now turn to the description and parametrization of the Riemannian metrics on $G/H$.

\begin{theorem}\label{MétInvRed}\cite[Chapter X, Corollary 3.2]{Nom2}
	The $G$-invariant Riemannian metrics on $G/H$ are in one to one correspondance with the $\Ad(H)$-invariant inner products on $\mathfrak{q}$.
\end{theorem}

Let $g$ be a $G$-invariant Riemannian metric on the flag manifold $G/H$ and let $g_\Cmb : \qfr^{\Cmb} \times \qfr^{\Cmb} \rightarrow \Cmb$ be the complexification of $g$, which is a non-degenerate simetric $\Cmb$-bilineal tensor on $\qfr^\Cmb$. If $T \seq H$ is a maximal torus of $G$, then the inner product $g$ on $\qfr$ is $\Ad(T)$-invariant. Hence
$$(\al(H)+\beta(H)) \hspace{1mm} g_\mathbb{C}(E_\al, E_\beta) = 0$$
for all $\al, \beta \in \Delta_\qfr$ and $H \in \tfr^\Cmb$. This implies that $\gfr_\al$ and $\gfr_\beta$ are $g_\Cmb$-orthogonal whenever $\al+\beta \neq 0$. Given an ordering $\Delta^+ \seq \Delta$, the inner product $g : \qfr \times \qfr \rightarrow \Rmb$ may be written as
$$g = \sum_{\al \in \Delta_\qfr^+} x_{\al} (-\kil_{\gfr})\rvert_{\mathfrak{q}_{\al}},$$
where the coefficients $x_{\al} := g(e_{\al}, e_\al) > 0$ satisfy the relations $x_\al = x_\beta$ whenever $\al \sim \beta$, since the metric is $\Ad(H)$-invariant (see \cite[Chapter 7, \S 6]{Arv}). It holds that any $G$-invariant Riemannian metric $g \equiv (x_{\al})_{\al \in \Delta_\qfr^+}$ on $G/H$ is almost Hermitian with respect to any $G$-invariant almost complex structure $J \equiv (\epsilon_{\al})_{\al \in \Delta_\qfr}$. Among the $G$-invariant Riemannian metrics, it is worth highlighting the \textit{standard metric}, which is given by $g_o := -\kil_{\gfr}$, i.e., $x_\al = 1$ for all $\al \in \Delta_{\qfr}^+$.  Additionally, we have the following characterization for Kähler metric on the complex flag manifold $(G/H, J)$.

\begin{theorem}\cite[Proposition 7.7]{Arv}
	The $G$-invariant Riemannian metric $g \equiv (x_\al)_{\al \in \Delta_{\qfr}^+}$ is a K\"ahler metric if and only if
	$x_{\al+ \beta} = x_{\al} + x_{\beta}$
	for all $\al, \beta, \al+\beta \in \Delta_{\qfr}^+$.
\end{theorem}

\subsection{$G$-invariant connections}\label{Sec2,3}
A connection $\nabla$ on a Hermitian manifold $(M,g,J)$ is said to be $\emph{Hermitian}$ if both the Riemannian metric and the complex structure are $\nabla$-parallel, i.e., $\nabla g = 0$, $\nabla J = 0$. Given a Hermitian structure $(g,J)$ on $M$, it is well known that there exists a mono-parametric family of Hermitian connections. This curve of connections was first studied by Gauduchon in \cite{Gaudu} and is given by
$$g(\nabla^t_X Y, Z) = g(\nabla^{g}_X Y, Z) - \tfrac{t-1}{4}d\omega(JX,JY,JZ) - \tfrac{t+1}{4}d\omega(JX,Y,Z)$$
for all $X, Y, Z \in \Xfr(M)$, where $\nabla^g$ is the Levi-Civita connection and $\omega$ is the \textit{fundamental 2-form} associated to $(g,J)$, i.e., $\omega := g(J\cdot, \cdot)$. These are called the \textit{canonical connections}. The \textit{Chern connection} $\nabla^1 =: \nabla^C$ and the \textit{Bismut connection} $\nabla^{-1} =: \nabla^B$ are of particular interest in the field of Hermitian geometry. The Bismut connection is the unique Hermitian connection on $(M,g,J)$ with totally skew-symmetric torsion, i.e., $\theta(X,Y,Z) := g(T^B(X,Y),Z)$) is a 3-form.

For flag manifolds $M = G/H$, if $(g,J)$ is a $G$-invariant Hermitian structure on $G/H$, then the canonical connections are all $G$-invariant. Recall that if $\nabla$ is such a connection, then it is completely characterized by the corresponding \textit{Nomizu operator} $\Lambda: \qfr \times \qfr \rightarrow \qfr$, $$\Lambda(X, Y) := \Lambda(X)Y := (\nabla_{X^*}Y^*)_{o} + [X,Y]_{\qfr} = (\nabla_{X^*}Y^* - [X^*, Y^*])_{o}$$ for $X,Y \in \qfr$ (see \cite[Chapter X]{Nom2}). We consider the action of the Nomizu operator on the $G$-invariant tensors $\varphi$ on $G/H$, given by
$(\Lambda(X) \cdot \varphi) := (\nabla_{X^*} \varphi)$
for $X \in \qfr$. Note that $\varphi$ is $\nabla$-parallel if and only if $\Lambda(X) \cdot \varphi \equiv 0$ for all $X \in \qfr$. Since both $g$ and $J$ are $\nabla$-parallel, we deduce that $\Lambda(X) : \qfr \rightarrow \qfr$ is a skew-Hermitian linear operator with respect to $g$ for all $X \in \qfr$. In terms of the Nomizu operator, the torsion and curvature tensors may be expressed as
\begin{align*}
	T(X,Y) &\equiv \Lambda(X)Y - \Lambda(Y)X - [X,Y]_\qfr, \\
	R(X,Y) &\equiv \Lambda(X)\circ \Lambda(Y) - \Lambda(Y)\circ \Lambda(X) - \Lambda([X,Y]_\qfr) - \ad([X,Y]_\hfr)\rvert_{\qfr}
\end{align*}
for all $X,Y \in \qfr$, where $[\cdot, \cdot]_\qfr$ and $[\cdot, \cdot]_\hfr$ denote the projections of the Lie bracket over $\qfr$ and $\hfr$, respectively. The actions of $\Lambda$ on $T$ and $R$ are given by
\begin{align*}
	(\Lambda(X) \cdot T)(Y,Z) &= \Lambda(X)(T(Y,Z)) - T(\Lambda(X)Y,Z) - T(Y, \Lambda(X)Z), \\
	(\Lambda(X) \cdot R)(Y,Z) &= [\Lambda(X), R(Y,Z)]  - R(\Lambda(X)Y,Z) -R(Y, \Lambda(X)Z) 
\end{align*}
for all $X,Y,Z,W \in \qfr$.

\section{The isotropy representation of a flag manifold}\label{Sec3}

In this section, we will derive explicit formulas for the number of irreducible $\Ad(H)$-invariant subspaces for the isotropy representation of a flag manifold $G/H$, usually denoted by $r = r(G/H)$, where $G$ is a compact simple Lie group. Recall that those subspaces are in one to one correspondence with the equivalence classes in $\Delta_\qfr^+$ under the relation $\al \sim \beta$ if and only if $\al\rvert_{\zfr(\hfr)} = \beta\rvert_{\zfr(\hfr)}$. Note that on any equivalence class there necessarily exist a highest root and a lowest root, both unique.

From here onwards, we will denote the root system of $\gfr = \mathrm{Lie}(G)$ by $\Delta$ and assume that $\Pi := \{ \al_1, \dots, \al_n\}$ is a basis for $\Delta$. Let $\Pi_\qfr := \Pi \cap \Delta_\qfr$ be the complementary simple roots and $\Pi_{\hfr} := \Pi \cap \Delta_\hfr$ be the corresponding basis for $\Delta_\hfr$. We will use the notation $b := \lvert \Pi_\qfr \rvert$ for the number of complementary simple roots and we will always consider $\Pi_\qfr = \{\al_{i_1}, \dots, \al_{i_b}\}$ with $1 \leq i_1 < \cdots < i_b \leq n$. A diagram may be constructed from the Dynkin diagram of $\Delta$ by painting the nodes of the roots in $\Pi_\hfr$ black and leaving the nodes of the roots in $\Pi_\qfr$ white. Said graph is called a \textit{painted Dynkin diagram}. It is a well known fact that there is a one to one correspondence between the diffeomorphism classes of flag manifolds and the painted Dynkin diagrams. (see \cite{Alek})

In the following theorems, we will make use of the descriptions of the positive roots for the classical systems in terms of the simple ones. Those may be found in \cite[p. 250-258]{Bourbaki}.

\begin{theorem}\label{FormulaAn}
	Suppose that $\Delta$ is of $A_n$ type. Then $r = \tbinom{b+1}{2}$.
\end{theorem}

\begin{proof}
	\begin{figure}[h]
		\centering
		\begin{tikzpicture}[scale = 1, baseline={(0,-0.1)}]
			\tikzstyle{every node}=[circle, draw, inner sep=1.5pt];
			
			\node (a1) at (0,0) {};
			\node (a2) at (1,0) {};
			\node[draw = none] (dots) at (1.5,0) {$\cdots$};
			\node (a3) at (2,0) {};
			\node (a4) at (3,0) {};
			
			\node[draw = none] at ($(a1.south)+(0,-8pt)$) {$\al_{1}$};
			\node[draw = none] at ($(a2.south)+(0,-8pt)$) {$\al_{2}$};
			\node[below=2pt of dots, draw=none] {};
			\node[draw = none] at ($(a3.south)+(0,-8pt)$) {$\al_{n-1}$};
			\node[draw = none] at ($(a4.south)+(0,-8pt)$) {$\al_{n}$};
			
			\draw (a1)--(a2) (a3)--(a4);
		\end{tikzpicture}
		\caption{Dynkin diagram of $A_n$ type}\label{FigureAn}
	\end{figure}
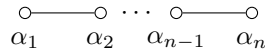
	Fix $n \in \Nmb$. The positive roots $\al \in \Delta^+$ are given by
	$$\al= \sum\limits_{s = i}^{j} \al_s$$
	for some $1 \leq i \leq j \leq n$. Given a root $\al \in \Delta_\qfr^+$ with $\al= \sum\limits_{s = i}^{j} \al_s$, note that there exist unique indexes $i \leq i_k \leq i_\ell \leq j$ such that
	$$\al \sim \beta_{k ,\ell} := \sum\limits_{s = i_k }^{i_\ell} \al_s.$$
	Said $\beta_{k, \ell} \in \Delta_{\qfr}^+$ is the lowest root in the equivalence class of $\al$. Hence the equivalence classes in $\Delta_\qfr^+$ under $\sim$ are in one to one correspondance with the pairs $(k, \ell)$ where $1 \leq k \leq \ell \leq b$. We conclude that $r = \tbinom{b+1}{2}$.
\end{proof}

\begin{theorem}\label{FormulaBn}
	Suppose that $\Delta$ is of $B_n$ type and denote $c := \lvert \{ i \colon \al_i \in \Pi_\hfr, \al_{i+1} \in \Pi_\qfr \} \rvert$. Then $r = b^2 + c$.
\end{theorem}

\begin{proof}
	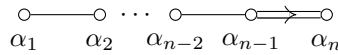
\begin{figure}[h]
		\centering
		\begin{tikzpicture}[scale = 1, baseline={(0,-0.1)}]
			\tikzstyle{every node}=[circle, draw, inner sep=1.5pt];
			
			\node (a1) at (0,0) {};
			\node (a2) at (1,0) {};
			\node[draw = none] (dots) at (1.5,0) {$\cdots$};
			\node (a3) at (2,0) {};
			\node (a4) at (3,0) {};
			\node (a5) at (4,0) {};
			
			\node[draw = none] at ($(a1.south)+(0,-8pt)$) {$\al_{1}$};
			\node[draw = none] at ($(a2.south)+(0,-8pt)$) {$\al_{2}$};
			\node[below=2pt of dots, draw=none] {};
			\node[draw = none] at ($(a3.south)+(0,-8pt)$) {$\al_{n-2}$};
			\node[draw = none] at ($(a4.south)+(0,-8pt)$) {$\al_{n-1}$};
			\node[draw = none] at ($(a5.south)+(0,-8pt)$) {$\al_{n}$};
			
			\draw (a1)--(a2) (a3)--(a4);
			\draw[double distance=2pt] (a4) -- (a5);
			\draw[double distance=2pt,-implies] (a4)-- ($(a5)!0.4!(a4)$); 
		\end{tikzpicture}
		\caption{Dynkin diagram of $B_n$ type}\label{FigureBn}
	\end{figure}
	Fix $n \in \Nmb$. The positive roots $\al \in \Delta^+$ are either given by
	\benalp
	\item $\al = \sum\limits_{s = i}^{j}\al_s$
	for some $1 \leq i \leq j \leq n$, or 
	\item $\al = \sum\limits_{s = i}^{j-1}\al_s + 2\sum\limits_{t = j}^{n}\al_t$
	for some $1 \leq i < j \leq n$.
	\een
	
	Suppose now that $\al$ is a positive complementary root. If $\al = \sum\limits_{s = i}^{j}\al_s$, then as in the proof of Theorem \ref{FormulaAn}, there exist unique indexes $i \leq i_k \leq i_\ell \leq j$ such that
	$$\al \sim \beta_{k, \ell}^1 := \sum\limits_{s = i_k}^{i_\ell}\al_s.$$
	
	If $\al = \sum\limits_{s = i}^{j-1}\al_s + 2\sum\limits_{t = j}^{n}\al_t$, then we may assume that at least one of the roots $\al_j, \dots, \al_n$ is complementary (otherwise, it holds that $\al \sim \sum\limits_{s = i}^{j-1} \al_s$; we have already accounted for such an equivalence class). Then there exists some $j \leq i_\ell \leq n$ such that
	$$\al \sim \al' := \sum\limits_{s = i}^{i_\ell-1}\al_s + 2\sum\limits_{t = i_\ell}^{n}\al_t.$$
	If at least one of the roots $\al_i, \al_{i+1}, \dots, \al_{i_\ell-1}$ is complementary, then there exist unique indexes $i \leq i_k < i_\ell \leq n$ such that
	$$\al \sim \beta^2_{k, \ell} := \sum\limits_{s = i_k}^{i_\ell-1}\al_s + 2\sum\limits_{t = i_\ell}^{n}\al_t.$$	
	If $\al_i, \al_{i+1}, \dots, \al_{i_\ell-1}$ are all non complementary, then
	$$\al \sim \beta^3_{\ell} := \al_{i_\ell -1} + 2\sum\limits_{t = i_\ell}^{n} \al_t.$$
	
	In either case, the roots $\beta^1_{k, \ell}$, $\beta^2_{k, \ell}$, $\beta^3_{\ell}$ are the lowest in their respective equivalence classes. Hence the total number of equivalence classes is
	\begin{align*}
		r &= \lvert\{(k, \ell) \colon 1 \leq k \leq \ell \leq b \} \rvert + \lvert \{ (k, \ell) \colon 1 \leq k < \ell \leq b \} \rvert + \lvert \{i \colon \al_i \in \Pi_\hfr, \al_{i+1} \in \Pi_\qfr \} \rvert \\
		&= \tbinom{b+1}{2} + \tbinom{b}{2} + c = b^2 + c.
	\end{align*}
	Thus we conclude the proof.
\end{proof}

\begin{theorem}\label{FormulaCn}
	Suppose that $\Delta$ is of $C_n$ type. Then $r = b^2$ if $\al_n \in \Pi_\qfr$, and $r = b^2+b$ if $\al_n \in \Pi_\hfr$. 
\end{theorem}

\begin{proof}
	\begin{figure}[h]
		\centering
		\begin{tikzpicture}[scale = 1, baseline={(0,-0.1)}]
			\tikzstyle{every node}=[circle, draw, inner sep=1.5pt];
			
			\node (a1) at (0,0) {};
			\node (a2) at (1,0) {};
			\node[draw = none] (dots) at (1.5,0) {$\cdots$};
			\node (a3) at (2,0) {};
			\node (a4) at (3,0) {};
			\node (a5) at (4,0) {};
			
			\node[draw = none] at ($(a1.south)+(0,-8pt)$) {$\al_{1}$};
			\node[draw = none] at ($(a2.south)+(0,-8pt)$) {$\al_{2}$};
			\node[below=2pt of dots, draw=none] {};
			\node[draw = none] at ($(a3.south)+(0,-8pt)$) {$\al_{n-2}$};
			\node[draw = none] at ($(a4.south)+(0,-8pt)$) {$\al_{n-1}$};
			\node[draw = none] at ($(a5.south)+(0,-8pt)$) {$\al_{n}$};
			
			\draw (a1)--(a2) (a3)--(a4);
			\draw[double distance=2pt] (a4) -- (a5);
			\draw[double distance=2pt,-implies] (a5)-- ($(a5)!0.5!(a4)$); 
		\end{tikzpicture}
		\caption{Dynkin diagram of $C_n$ type}\label{FigureCn}
	\end{figure}
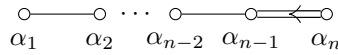
	Fix $n \in \Nmb$. The positive roots $\al \in \Delta^+$ are either given by
	\benalp
		\item $\al = \sum\limits_{s = i}^{j} \al_s$ for some $1 \leq i \leq j \leq n$, or
		\item $\al = \sum\limits_{s = i}^{j-1}\al_s + 2\sum\limits_{t = j}^{n-1} \al_t + \al_n$
		for some $1 \leq i \leq j \leq n-1$.
	\een
	Suppose now that $\al \in \Delta_\qfr^+$. If $\al= \sum\limits_{s = i}^{j} \al_s$, then there exist, as in Theorem \ref{FormulaAn}, unique indexes $i \leq i_k \leq i_\ell \leq j$ such that
	$$\al \sim \beta_{k, \ell}^1 := \sum\limits_{s = i_k}^{i_\ell} \al_s.$$
	Hence to the equivalence class of $\al$ corresponds a unique pair $(k, \ell)$ with $1 \leq k \leq \ell \leq b$.
		
	If $\al = \sum\limits_{s = i}^{j-1}\al_s + 2\sum\limits_{t = j}^{n-1} \al_t + \al_n$, then we may assume that some root $\al_j, \dots, \al_{n-1}$ is complementary (otherwise, it holds that $\al \sim \sum\limits_{s = i}^{n}\al_s$). Then there exist unique indexes $i \leq i_k \leq i_\ell \leq n-1$ such that
	$$\al \sim \beta^2_{k, \ell} := \sum\limits_{s = i_k}^{i_\ell-1}\al_s + 2\sum\limits_{t = i_\ell}^{n-1} \al_t + \al_n.$$
	Hence to the equivalence class of $\al$ corresponds a unique pair $(k, \ell)$ such that $1 \leq k \leq \ell \leq b$ if $\al_n \in \Pi_\hfr$, and $1 \leq k \leq \ell \leq b-1$ if $\al_n \in \Pi_\qfr$.
	
	We conclude that the number of equivalence classes is given by
$$r = \begin{cases}
	\lvert \{ (k, \ell) \colon 1 \leq k \leq \ell \leq b \}\rvert + \lvert \{(k, \ell) \colon 1 \leq k \leq \ell \leq b-1 \}\rvert = \tbinom{b+1}{2} + \tbinom{b}{2} = b^2, \text{ if } \al_n \in \Pi_\qfr, \\
	2\lvert \{ (k, \ell) \colon 1 \leq k \leq \ell \leq b \}\rvert = 2\tbinom{b+1}{2} = b^2 + b, \text{ if } \al_n \in \Pi_\hfr.
\end{cases}$$
Thus the theorem has been proved.
\end{proof}

\begin{theorem}\label{FormulaDn}
	Suppose that $\Delta$ is of $D_n$ type. Then
	$$r = \begin{cases}
		b^2+c, &\text{if } \al_{n-1}, \al_n \in \Pi_\hfr, \\
		b^2-b+c, &\text{if } \al_{n-2},\al_{n-1}, \al_n \in \Pi_\qfr, \\
		b^2-b+1+c, &\text{otherwise, }
	\end{cases}$$
	where $c := \lvert \{ i \in \{1, \dots, n-3 \} \colon \al_i \in \Pi_\hfr, \al_{i+1} \in \Pi_\qfr \} \rvert$.
\end{theorem}

\begin{proof}
	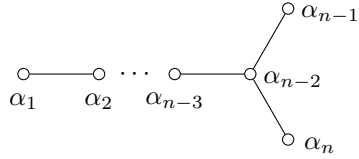
\begin{figure}[h]
		\centering
		\begin{tikzpicture}[scale = 1, baseline={(0,-0.1)}]
			\tikzstyle{every node}=[circle, draw, inner sep=1.5pt];
			
			\node (a1) at (0.25,0) {};
			\node (a2) at (1.25,0) {};
			\node[draw = none] (dots) at (1.75,0) {$\cdots$};
			\node (a3) at (2.25,0) {};
			\node (a4) at (3.25,0) {};
			\node (a5) at (3.75,0.866) {};
			\node (a6) at (3.75,-0.866) {};
			
			\node[draw = none] at ($(a1.south)+(0,-8pt)$) {$\al_1$};
			\node[draw = none] at ($(a2.south)+(0,-8pt)$) {$\al_2$};
			\node[draw = none] at ($(a3.south)+(0,-8pt)$) {$\al_{n-3}$};
			\node[draw = none] at ($(a4.south)+(16pt,0)$) {$\al_{n-2}$};
			\node[below=2pt of dots, draw=none] {};
			\node[draw = none] at ($(a5.south)+(16pt,0)$) {$\al_{n-1}$};
			\node[draw = none] at ($(a6.south)+(12pt,0)$) {$\al_n$};
			
			\draw (a1)--(a2) (a3)--(a4) (a4)--(a5) (a4)--(a6);
		\end{tikzpicture}
		\caption{Dynkin diagram of $D_n$ type}\label{FigureDn}
	\end{figure}
	Fix $n \in \Nmb$. The positive roots $\al \in \Delta^+$ are either of the form
	\benalp
		\item $\al = \sum\limits_{s = i}^{j}\al_s$ for some $1 \leq i \leq j \leq n-1$,
		\item $\al = \sum\limits_{s = i}^{n-2}\al_s + \al_n$ for some $1 \leq i \leq n$, $i \neq n-1$,
		\item $\al = \sum\limits_{s = i}^{n}\al_s$ for some $1 \leq i \leq n-2$, or
		\item $\al= \sum\limits_{s = i}^{j-1} \al_s + 2\sum\limits_{t = j}^{n-2} \al_t + \al_{n-1} + \al_n$ for some $1 \leq i < j \leq n-2$.
	\een
	Suppose now that $\al \in \Delta_\qfr^+$. If $\al = \sum\limits_{s = i}^{j}\al_s$, then there exist, as in Theorem \ref{FormulaAn}, unique indexes $i \leq i_k \leq i_\ell \leq j$ such that
	$$\al \sim \beta_{k, \ell}^1 := \sum\limits_{s = i_k}^{i_\ell}\al_s.$$
	Hence to the equivalence class of $\al$ corresponds a unique pair $(k, \ell)$ with $1 \leq k \leq \ell \leq b$ if $\al_n \in \Pi_\hfr$, and with $1 \leq k \leq \ell \leq b-1$ if $\al_n \in \Pi_\qfr$.
		
	If $\al = \sum\limits_{s = i}^{n-2}\al_s + \al_n$ and $\al_n \in \Pi_\hfr$, then $\al \sim \sum\limits_{s = i}^{n-2}\al_s$. Otherwise, there exists a unique $i \leq i_k \leq n$ with $i_k \neq n-1$, such that
	$$\al \sim \beta_{k}^2 := \al = \sum\limits_{s = i_k}^{n-2}\al_s + \al_n.$$	
	Hence to the equivalence class of $\al$ corresponds a unique index $1 \leq k \leq b$ if $\al_{n-1} \in \Pi_\hfr$ and $\al_n \in \Pi_\qfr$, and an index $1 \leq k \leq b$, $k \neq b-1$ if $\al_{n-1}, \al_n \in \Pi_\qfr$.
		
	If $\al = \sum\limits_{s = i}^{n}\al_s$ and $\al_n \in \Pi_\hfr$, then $\al \sim \sum\limits_{s = i}^{n-1}\al_s$. In said scenario, we need not account for the roots of $(b)$ type. We may proceed in a similar manner if $\al_{n-1} \in \Pi_\hfr$. Assume now that both $\al_{n-1}$ and $\al_n$ are complementary. If at least one of the roots $\al_i, \al_{i+1}, \dots, \al_{n-2}$ is complementary, then there exists a unique index $i \leq i_k \leq n-2$ such that 
	$$\al \sim \beta^{3}_{k} := \sum\limits_{s = i_k}^{n}\al_s.$$
	To such an $\al$ corresponds a unique index $1 \leq k \leq b-2$. If $\al_{i}, \al_{i+1},\dots, \al_{n-2}$ are all non complementary, then $\al \sim \al_{n-2}+\al_{n-1}+\al_n$. Note that this scenario is only possible when $\al_{n-2} \in \Pi_\hfr$.
		
	If $\al = \sum\limits_{s = i}^{j-1} \al_s + 2\sum\limits_{t = j}^{n-2} \al_t + \al_{n-1} + \al_n$, then we may suppose that at least one of the roots $\al_{j}, \dots, \al_{n-2}$ is complementary (otherwise, it holds that $\al \sim \sum\limits_{s = i}^{n}\al_s$). Therefore, there exists a unique index $j \leq i_\ell \leq n-2$ such that 
	$$\al \sim \sum\limits_{s = i}^{i_\ell-1} \al_s + 2\sum\limits_{t = i_\ell}^{n-2} \al_t + \al_{n-1} + \al_n.$$ If at least one of the roots $\al_i, \dots, \al_{i_\ell -1}$ is a complementary root, then there exist unique indexes $i \leq i_k < i_\ell \leq n-2$ such that
	$$\al \sim \beta_{k, \ell}^4 := \sum\limits_{s = i_k}^{i_\ell-1} \al_s + 2\sum\limits_{t = i_\ell}^{n-2} \al_t + \al_{n-1} + \al_n.$$
	To such an $\al$ corresponds a unique pair $(k, \ell)$ with $1 \leq k < \ell \leq b$ if $\al_{n-1}, \al_n$ are both non complementary roots, with $1 \leq k < \ell \leq b-1$ if only of them is a complementary root, and with $1 \leq k < \ell \leq b-2$ if both $\al_{n-1}, \al_n$ are complementary roots. If $\al_i, \dots, \al_{i_\ell -1}$ are all non complementary, then
	$$\al \sim \beta_{\ell}^5 := \al_{i_\ell -1} + 2\sum\limits_{t = i_\ell}^{n-2}\al_t +\al_{n-1}+\al_n.$$

	In each case, the complementary roots described are the lowest representatives of the equivalence class of $\al$.
	\begin{enumerate}
		\item If $\al_{n-1}, \al_n \in \Pi_\hfr$, then the total number of equivalence classes is given by 
		$$r = \tbinom{b+1}{2} + \tbinom{b}{2}+c = b^2 + c.$$
		
		\item If $\al_{n-2},\al_{n-1},\al_n \in \Pi_\qfr$, then
		$$r = \tbinom{b}{2}+(b-1) + (b-2) + \tbinom{b-2}{2}+c = b^2-b+c.$$
		
		\item For the remaining cases, it holds that $r = b^2-b+1+c$.
	\end{enumerate}
\end{proof}

We provide the following table of formulas for $r := r(G/H)$ in terms of the amount of complementary simple roots $b := \lvert \Pi_\qfr \rvert$. Recall that we have defined 
$$c := \begin{cases}
	\lvert \{i \in \{1, \dots, n-1 \} \colon \al_i \in \Pi_\hfr, \al_{i+1} \in \Pi_\qfr \} \rvert, \hspace{2mm} \text{if } \Delta \text{ is of } B_n \text{ type}, \\
	
	\lvert \{i \in \{1, \dots, n-3 \} \colon \al_i \in \Pi_\hfr, \al_{i+1} \in \Pi_\qfr \} \rvert, \hspace{2mm} \text{if } \Delta \text{ is of } D_n \text{ type},
\end{cases}$$
and that the ordered bases $\Pi = \{\al_1, \dots, \al_n \}$ for $\Delta$ are the ones described in Figures \ref{FigureAn}, \ref{FigureBn}, \ref{FigureCn} and \ref{FigureDn}.

\begin{table}[ht]
	\centering
	\small
	\setlength{\tabcolsep}{4pt} 
	\renewcommand{\arraystretch}{1.4} 
	
	\makeatletter
	\newcommand{\lineagruesa}{\noalign{\global\arrayrulewidth=1.5pt}\hline\noalign{\global\arrayrulewidth=0.4pt}}
	\makeatother
	
	\caption{Number of irreducible isotropy summands of flag manifolds $G/H$, with $G$ a compact classical simple Lie group.}
	\label{tab:classical}
	\vspace{-3mm}
	
	\begin{tabular}{|c | c |}
		\lineagruesa
		Type for $\Delta$ & $r := r(G/H)$\\ 
		\lineagruesa
		$A_n$ & $\tbinom{b+1}{2}$\\ \hline
		$B_n$ & $b^2+c$ \\ \hline
		$C_n$ & $\begin{cases}
			b^2, &\text{if } \alpha_n \in \Pi_{\qfr}, \\
			b^2+b, &\text{if } \alpha_n \in \Pi_{\hfr}.
		\end{cases}$\\ \hline
		$D_n$ & $\begin{cases}
			b^2+c, &\text{if } \alpha_{n-1}, \alpha_n \in \Pi_{\hfr}, \\
			b^2-b+c, &\text{if } \alpha_{n-2}, \alpha_{n-1}, \alpha_n \in \Pi_{\qfr}, \\
			b^2-b+1+c, &\text{otherwise}.
		\end{cases}$ \\ \hline
	\end{tabular}
\end{table}

\begin{remark}\label{remark1}
On \cite[Sections 3, 4, 5, 6]{Alek}, several formulas for $r$ were given, for some families of flag manifolds $G/H$ with $G$ a compact classical simple Lie group. Those expressions agree with ours when $G$ is of $A_n$ type, but they differ in all other cases.
\end{remark}

Finally, by a hand calculation, we obtain the following list for the number of irreducible isotropy subspaces, for the flag manifolds of the exceptional root systems of $F_4$ and $G_2$ type.

\begin{table}[ht]
	\centering
	\small
	\setlength{\tabcolsep}{4pt} 
	
	\renewcommand{\arraystretch}{1.4}
	
	\makeatletter
	\newcommand{\lineagruesa}{\noalign{\global\arrayrulewidth=1.5pt}\hline\noalign{\global\arrayrulewidth=0.4pt}}
	\makeatother
	
	\caption{Number of irreducible isotropy summands for flag manifolds $G/H$, with $G$ a compact simple Lie group of $F_4$ type or of $G_2$ type.}
	\label{tab:expectF4andG2}
	\vspace{-3mm}
	
	\begin{tabular}{|c | c || c | c |}
		\lineagruesa
		Painted Dynkin diagram & $r := r(G/H)$ & Painted Dynkin diagram & $r := r(G/H)$ \\ 
		\lineagruesa
		\DynF{white}{white}{white}{white} & $24$ & \DynF{black}{white}{black}{white} & $8$ \\ \hline
		\DynF{white}{white}{white}{black} & $13$ & \DynF{black}{black}{white}{white} & $9$ \\ \hline
		\DynF{white}{white}{black}{white} & $13$ & \DynF{white}{black}{black}{black} & $2$ \\ \hline
		\DynF{white}{black}{white}{white} & $16$ & \DynF{black}{white}{black}{black} & $3$ \\ \hline
		\DynF{black}{white}{white}{white} & $16$ & \DynF{black}{black}{white}{black} & $4$ \\ \hline
		\DynF{white}{white}{black}{black} & $6$  & \DynF{black}{black}{black}{white} & $2$ \\ \hline
		\DynF{white}{black}{white}{black} & $8$ & \DynG{white}{white} & $6$ \\ \hline
		\DynF{black}{white}{white}{black} & $8$ & \DynG{white}{black} & $2$ \\ \hline
		\DynF{white}{black}{black}{white} & $6$ & \DynG{black}{white} & $3$ \\ \hline
	\end{tabular}
\end{table}

\section{BTP metrics on flag manifolds}\label{Sec4}

On this section, we will always assume $(M, g, J)$ to be a Hermitian manifold. Our object of study will be the parallelism conditions deriving from the \textit{Bismut connection} $\nabla^B$ (see \ref{Sec2,3}).

\begin{definition}
	The Hermitian manifold $(M,g,J)$ is said have \textit{parallel Bismut torsion} or to be a \textit{BTP manifold} if $T^B$ is $\nabla^B$-parallel, and it is said to be a \textit{Bismut-Ambrose-Singer} manifold or a \textit{BAS manifold} if both $T^B$ and $R^B$ are $\nabla^B$-parallel. Given a complex manifold $(M, J)$ and a Riemannian metric $g$ on $M$, $g$ is said to be a \textit{BTP metric} (respectively a \textit{BAS metric}) if the Hermitian manifold $(M,g,J)$ is a BTP manifold (respectively a BAS manifold).
\end{definition}

From here onwards, we will assume $(G/H, J)$ to be a complex flag manifold, with $J$ a fixed $G$-invariant complex structure on $G/H$ (see \ref{Sec2,2}). We will always assume $\Delta_\qfr^+ = \Delta^+ \cap \Delta_\qfr = P(J)$ to be the invariant ordering for $\Delta_\qfr$ defined by $J$. Our focus will be the study of $G$-invariant BAS metrics and BTP metrics on $(G/H, J)$ where $G$ is a compact simple Lie group. It can been proven that any K\"ahler metric and any multiple of the standard metric $g_o$ is a BAS metric on $(G/H, J)$ \cite[Proposition 4.1]{PodZ}.

\begin{theorem}\cite[Section 4]{PodZ}\label{EqNomizuTorsion}
	Let $g$ be a $G$-invariant Hermitian metric on $(G/H, J)$, $\nabla^B$ the Bismut connection and $\Lambda^B$ the corresponding Nomizu operator on $\qfr$. Then for any $\al, \beta, \al+\beta \in \Delta_\mathfrak{q}$, the following equations are satisfied:
		\begin{enumerate}[font = \normalfont, label = (\alph*)]
		\item $\Lambda^B(E_\al)(E_\beta) = N_{\al, \beta} \dfrac{(x_{\al+\beta}+x_\beta-x_{\al}-\epsilon_{\al}\epsilon_\beta\epsilon_{\al+\beta}(\epsilon_{\al}x_\al+\epsilon_\beta x_\beta - \epsilon_{\al+\beta}x_{\al+\beta}))}{2x_{\al+\beta}}E_{\al+\beta}$,
		\item $T^B(E_\al, E_\beta) = N_{\al, \beta} \dfrac{\epsilon_{\al}\epsilon_{\beta}\epsilon_{\al+\beta}(\epsilon_{\al+\beta}x_{\al+\beta}-\epsilon_{\al}x_{\al}-\epsilon_{\beta}x_{\beta})}{x_{\al+\beta}}E_{\al+\beta}$.
	\end{enumerate}
\end{theorem}

\begin{theorem}\label{BTP2summands}\cite[Theorem 1.3]{PodZ}
	Let $G$ be a compact simple Lie group such that $r(G/H) \leq 2$. Then every $G$-invariant BTP metric on $(G/H, J)$ is either a K\"ahler metric or a multiple of the standard metric.
\end{theorem}

\begin{theorem}\label{BTPFFSpecialUnitary}\cite[Theorem 1.4]{PodZ}
	Every $G$-invariant BTP metric on the complex full flag manifold $(SU(n+1)/T^n, J)$ is either a K\"ahler metric or a multiple of the standard metric.
\end{theorem}

At this very moment, there have been no other characterization results similar to Theorem \ref{BTP2summands} and Theorem \ref{BTPFFSpecialUnitary}, in the context of BTP flag manifolds. We will prove the following conjecture about the relationship between $G$-invariant BTP metrics and $G$-invariant BAS metrics, which was first formulated in \cite{PodZ}.

\begin{conjecture}\cite[Conjecture 1.6]{PodZ}
	Every $G$-invariant BTP metric on $(G/H, J)$ is either a K\"ahler metric or a multiple of the standard metric.
\end{conjecture}

In \cite{PodZ} several restrictions for $G$-invariant BTP metrics were obtained. From here onwards, we will assume $g$ to be a $G$-invariant BTP metric on $(G/H, J)$.

\begin{lemma}\label{Lemma1PodZ}\cite[Lemma 4.3, Proposition 4.4]{PodZ}
	Given $\al, \beta \in \Delta_{\qfr}^+$ complementary roots such that $\al+\beta \in \Delta_\mathfrak{q}^+$, the following statements are satisfied:
	
	\begin{enumerate}[font = \normalfont, label = (\alph*)]
		\item If $\al-\beta \notin \Delta_\mathfrak{q}$, then either $x_{\al+\beta} = x_\al + x_\beta$ or $x_{\al+\beta} = x_\al = x_\beta$.
		\item If $\al-\beta \in \Delta_\mathfrak{q}^+$, then either $x_{\al+\beta} = x_\al + x_\beta$ or $x_{\al+\beta} = x_\al$.
		\item If $\al-\beta \in \Delta_\mathfrak{q}^-$, then either $x_{\al+\beta} = x_\al + x_\beta$ or $x_{\al+\beta} = x_\beta$.
		\item Suppose that there exists some $\gamma \in \Delta_\mathfrak{q}^+$ such that $\beta + \gamma$, $\al+\beta+\gamma \in \Delta_\mathfrak{q}^+$ and $\al+\gamma \notin \Delta_\mathfrak{q}^+$. Then 
		\begin{equation}\label{eq1}
			\dfrac{(x_{\al+\beta+\gamma}-x_\al)(x_\beta+x_\gamma - x_{\beta+\gamma})}{x_{\beta+\gamma}} = \dfrac{(x_{\al+\beta}-x_\al)(x_{\al+\beta}+x_\gamma-x_{\al+\beta+\gamma})}{x_{\al+\beta}}.
		\end{equation}
	\end{enumerate}
\end{lemma}

We will present the following result, which is a generalization of Lemma \ref{Lemma1PodZ}.

\begin{lemma}\label{MyLemma2Roots}
    Given $\al, \beta \in \Delta_{\qfr}^+$ complementary roots such that $\al+\beta \in \Delta_{\qfr}^+$, it holds that either $x_{\al+\beta} = x_{\al}+x_{\beta}$ or $x_{\al+\beta} = x_{\al} = x_\beta$.
\end{lemma}

\begin{proof}
    The result follows from Lemma \ref{Lemma1PodZ} if $\al - \beta \notin \Delta_\qfr$. We now consider complementary roots $\al$ and $\beta$ such that $\al- \beta \in \Delta_\qfr$. Without loss of generality, we may assume that $\al - \beta \in \Delta_\qfr^+$ and that $x_{\al+\beta} = x_\al$. We will show that $x_{\al+\beta} = x_\beta$ holds.
    
    We define $\al_1 := \al - \beta$. By applying equation \eqref{eq1} to the roots $\al', \beta', \gamma'$, where $\al' := \beta$, $\beta' := \al_1$ and $\gamma' := \beta$, we deduce that
    $$(x_\al - x_\beta)(x_{\al_1}+x_\beta - x_\al) = (x_\al - x_\beta)x_\beta.$$
    Therefore, either $x_\al = x_\beta$ or $x_\al = x_{\al_1}$. For the latter case, we consider $\al_2 := \al_1 - \beta = \al-2\beta$. If $\al_2$ is not a positive complementary root, then it holds that $x_\al = x_{\al_1} + x_\beta$ or $x_\al = x_\beta$. Since $x_\al = x_{\al_1}$, we conclude that $x_\al = x_\beta$. If $\al_2 \in \Delta_\mathfrak{q}^+$, then as before we deduce that $x_{\al_1} = x_\beta$ or $x_{\al_1} = x_{\al_2}$. For the second case, we define $\al_3 := \al_2 - \beta$ and repeat the same reasoning.
        
    Since $\Delta_\mathfrak{q}^+$ is finite, this process must end, i.e., there exists some $n \in \mathbb{N}$ so that $\al_n := \al - n \beta \in \Delta_\mathfrak{q}^+$ and $\al_{n+1} := \al_n - \beta \notin \Delta_\mathfrak{q}^+$. For such an $n$, note that $$x_{\al+\beta} = x_\al = x_{\al_1} = \dots = x_{\al_n} = x_\beta.$$
    Hence the result has been proven.
\end{proof}

\newpage
\begin{definition} 
Given $A \subseteq \Delta_\mathfrak{q}^+$ nonempty, we say that:
	\begin{enumerate}[font = \normalfont, label = (\alph*)]
		\item $A$ is $\emph{K\"ahler}$ if $x_{\al+\beta} = x_\al+x_\beta$ for all $\al,\beta,\al+\beta \in A$. 
		\item $A$ is $\emph{Killing}$ if $x_{\al+\beta} = x_\al = x_\beta$ for all $\al, \beta, \al+\beta \in A$.
		\end{enumerate}
\end{definition}

The result we aim to prove is that $\Delta_\qfr^+$ is either a K\"ahler set or a Killing set, whichever the type of $\Delta$ may be. Firstly we will give examples of families of subsets of $\Delta_\qfr^+$, which are known to either be K\"ahler sets or Killing sets.
For $\al, \beta \in \Delta$, $\al \neq \pm \beta$, recall that the $\beta\emph{-series on }\al$ is defined as $\mathcal{S}_{\beta, \al} := \{ \al + n\beta \colon n \in \mathbb{Z} \} \cap \Delta$. From the study of root systems, it is known that $\mathcal{S}_{\beta, \al} = \{ \al+n\beta \colon r \leq n \leq s \}$ for some integers $r \leq 0 \leq s$. Observe that if $\al, \beta \in \Delta_\mathfrak{q}^+$ and $\al+n\beta \in \Delta$, then $\al+n\beta \in \Delta_\qfr^+$ for all $n \geq 0$.

\begin{lemma}\label{MyLemmaSeries}
    The set $(\mathcal{S}_{\beta, \al}\cup \{\beta\} )\cap \Delta_\mathfrak{q}^+$ is either a K\"ahler set or a Killing set for any $\al, \beta \in \Delta_\qfr^+$ such that $\al+\beta \in \Delta_\qfr^+$.
\end{lemma}

\begin{proof}
    Let $k \in \mathbb{N}\cup \{ 0\}$ be maximum such that $\al_k := \al - k \beta \in \Delta^+_\qfr$. Note that 
    $$\mathcal{S}_{\beta, \al} \cap \Delta^+_\qfr = \mathcal{S}_{\beta, \al'} \cap \Delta^+_\qfr.$$
    Therefore, we may assume that $k = 0$ (otherwise, we repeat the reasoning replacing $\al$ with $\al_k$). Then $\mathcal{S}_{\beta, \al}\cap \Delta_\mathfrak{q}^+ = \{ \al + n\beta \colon 0 \leq n \leq s \}$. By Lemma \ref{MyLemma2Roots}, we have that $x_{\al+\beta} = x_\al + x_\beta$ or $x_{\al+\beta} = x_\al = x_\beta$. Suppose that $x_{\al+\beta} = x_\al+x_\beta$. We may also assume that $s \geq 2$; otherwise, we are done. We will show, by strong induction on $n \in \mathbb{N}$, that $x_{\al+n\beta} = x_{\al}+nx_\beta$ for all $1 \leq n \leq s$.
    
    Suppose that the condition holds for all $n \in \{1, \dots, k \}$ and that $\al+(k+1)\beta \in \Delta_\qfr^+$. Applying equation \eqref{eq1} to $\al' := \beta$, $\beta' := \al+(k-1)\beta$, $\gamma' := \beta$ yields
    $$(x_{\al+(k+1)\beta} - x_\beta)(x_{\al+(k-1)\beta} + x_\beta - x_{\al+k\beta}) = (x_{\al+k\beta} - x_\beta)(x_{\al+k\beta}+x_\beta - x_{\al+(k+1)\beta}).$$
    By our hypothesis on $x_{\al+k\beta}$, we deduce that $x_{\al+(k+1)\beta} = x_{\al+k\beta}+x_\beta = x_{\al} +(k+1)x_\beta$. Hence $(\mathcal{S}_{\beta, \al}\cup \{\beta\} )\cap \Delta_\mathfrak{q}^+$ is a K\"ahler set. From the assumption $x_{\al+\beta} = x_{\al} = x_\beta$, it follows $(\mathcal{S}_{\beta, \al}\cup \{\beta\} )\cap \Delta_\mathfrak{q}^+$ is a Killing set, in much the same way.
\end{proof}

\begin{definition}
    We say that two sets $A, B \subseteq \Delta_\mathfrak{q}^+$ are $\emph{overlapping}$ if there exist $\al,\beta, \al+\beta \in A\cap B$.
\end{definition}

\begin{lemma}\label{MyLemmaOverlap}
    Let $A, B \subseteq \Delta_\mathfrak{q}^+$ be overlapping such that $B$ is K\"ahler or Killing. 
    \begin{enumerate}[font = \normalfont, label = (\alph*)]
        \item If $A$ is K\"ahler, then $B$ is K\"ahler.
        \item If $A$ is Killing, then $B$ is Killing.
    \end{enumerate}
\end{lemma}

\begin{proof}
Suppose that $A$ is K\"ahler. There exist $\al, \beta, \al+\beta \in A\cap B \subseteq A$, for which $x_{\al+\beta} = x_\al + x_\beta$ holds. Since those roots are also elements of $B$, we deduce that $B$ cannot be Killing. Hence $B$ must be K\"ahler. The proof when $A$ is Killing follows in a similar manner.
\end{proof}

\begin{definition}
    Let $C = \{ \al_1, \dots, \al_n \} \subseteq \Delta^+$ be an ordered set with possible repetitions. Then $C$ is said to be $A_n$-$\emph{like}$ if for any $k$-uple $i_1 < \dots < i_k$, it holds that $\al_{i_1} + \cdots + \al_{i_k} \in \Delta^+$ if and only if $i_1, \dots, i_k$ are consecutive.
    
    For $C$ an $A_n$-like set, we define $\tri{C} := \left\{ \sum\limits_{i = j}^{k} \al_i \colon 1 \leq j \leq k \leq n\right\}$ and $\tri{C, \qfr} := \tri{C}\cap \Delta_\qfr \seq \Delta_\qfr^+$. The set $\tri{C}$ is called the \textit{triangle} of $C$ or the \text{triangle with base} $C$.
\end{definition}

The $A_n$-like sets represent an important family of subsets of $\Delta^+$ as the following result shows.

\begin{lemma}\label{MYLEMMA}
    Let $C = \{ \al_1, \dots, \al_n \} \subseteq \Delta^+$ be $A_n$-like. Then exactly one of the following statements is true:
    \begin{enumerate}[font = \normalfont, label = (\alph*)]
        \item $x_{\al} = \sum\limits_{i} x_{\al_i}$ for all $\al \in \tri{C, \qfr}$, $\al = \sum\limits_{i} \al_i + \sum\limits_{j} \beta_j$, with $\al_i \in C\cap \Delta_\mathfrak{q}$, $\beta_j \in C \cap \Delta_\mathfrak{h}$.
        \item $x_\al = x_\beta$ for all $\al,
        \beta \in \tri{C, \qfr}$.
    \end{enumerate}
    In particular, $\tri{C, \qfr}$ is K\"ahler or Killing.
\end{lemma}

\begin{proof}
    Suppose that $C \cap \Delta_\qfr = \{ \al_{i_1}, \dots, \al_{i_b}\} \neq \emptyset$ so that $1 \leq i_1 < \cdots < i_b \leq n$. Define 
    $$\beta_1 := \sum\limits_{s = 1}^{i_1} \al_s, \hspace{2mm} \beta_b := \sum\limits_{s = i_{b-1}+1}^{n}\al_s, \hspace{2mm} \beta_{j} := \sum\limits_{i_{j-1}+1}^{i_j}\al_s$$
    for $2 \leq j \leq b-1$, and $C' := \{\beta_1, \dots, \beta_b\}$. Note that $C'$ is an $A_b$-like set such that $\tri{C'} \seq \Delta_\qfr^+$. Therefore, without loss of generality, we may assume that $C \seq \Delta_\qfr^+$. We then proceed by induction on $n = |C| \geq 1$. There is nothing to prove when $n = 1$, and the result when $n = 2$ follows from Lemma \ref{MyLemma2Roots}. 
    
    Suppose that $C = \{ \al, \beta, \gamma \}$ satisfies that $\al+\beta$, $\beta+\gamma$, $\al+\beta+\gamma \in \Delta$, but $\al+\gamma \notin \Delta$. By applying Lemma \ref{MyLemma2Roots} to the possible sums of pairs, we obtain the following equations
    \[
    \begin{cases}
    	x_{\al+\beta} = x_\al+x_\beta  \text{ or } x_{\al+\beta} = x_\al = x_\beta, \\
    	
    	x_{\beta+\gamma} = x_\beta+x_\gamma \text{ or } x_{\beta+\gamma} = x_\beta = x_\gamma, \\
    	
    	x_{\al+\beta+\gamma} = x_{\al+\beta}+x_\gamma \text{ or } x_{\al+\beta+\gamma} = x_{\al+\beta} = x_\gamma, \\
    	
    	x_{\al+\beta+\gamma} = x_\al+x_{\beta+\gamma} \text{ or } x_{\al+\beta+\gamma} = x_\al = x_{\beta+\gamma}.
    \end{cases}
    \]
    If $x_{\al+\beta+\gamma} = x_{\al+\beta}+x_\gamma = x_{\al} = x_{\beta+\gamma}$, we would have $x_\al \leq x_{\al+\beta} < x_{\al+\beta+\gamma} = x_{\al}$. The same reasoning may be applied if $x_{\al+\beta+\gamma} = x_{\al+\beta} =x_\gamma = x_{\al} + x_{\beta+\gamma}$. The system can then be simplified
    \[
    \begin{cases}
    	x_{\al+\beta} = x_\al+x_\beta \text{ or } x_{\al+\beta} = x_\al = x_\beta, \\
    	
    	x_{\beta+\gamma} = x_\beta+x_\gamma \text{ or } x_{\beta+\gamma} = x_\beta = x_\gamma, \\
    	
    	x_{\al+\beta+\gamma} = x_{\al+\beta}+x_\gamma = x_\al + x_{\beta+\gamma} \text{ or } x_{\al+\beta+\gamma} = x_{\al+\beta} = x_\gamma = x_{\al} = x_{\beta+\gamma}.
    \end{cases}
    \]
    If $x_{\al+\beta + \gamma} = x_{\al+\beta} + x_\gamma$, then by applying equation \eqref{eq1} to $\al, \beta, \gamma$ we see that $x_{\beta+\gamma} = x_\beta + x_\gamma$. Then $x_{\al+\beta} + x_\gamma = x_\al+x_\beta+x_\gamma$, and therefore $x_{\al+\beta} = x_\al+x_\beta$. For the case when $x_{\al+\beta+\gamma} = x_{\al+\beta} = x_\gamma = x_{\al} = x_{\beta+\gamma}$, since $x_{\al+\beta} = x_\al$, we deduce that all six coefficients must be equal.
        
    We assume that the result holds for any $A_{n-1}$-like set for some $n \geq 4$. Let $C = \{ \al_1, \dots, \al_{n} \}$ be $A_{n}$-like. Since $C_1 := \{ \al_1, \dots, \al_{n-1} \}$ and $C_2 := \{\al_2, \dots, \al_{n} \}$ are $A_{n-1}$-like, by the induction hypothesis, $\tri{C_1, \qfr}, \tri{C_2, \qfr}$ each satisfy either statement (a) or (b). Since $\tri{C_1, \qfr}$ and $\tri{C_2, \qfr}$ are overlapping, then by Lemma \ref{MyLemmaOverlap} they both satisfy (a) or both satisfy (b). Observe that $\tri{C_1, \qfr} \cup \tri{C_2, \qfr} = \tri{C} - \{\beta\}$ where $\beta : = \sum\limits_{i = 1}^{n} \al_i$. Consider $C_3 := \{\al_1, \gamma, \al_n \}$, where $\gamma := \sum\limits_{i = 2}^{n-1} \al_i \in \Delta_\mathfrak{q}^+$. Since $C_3$ is $A_3$-like and $\tri{C_3, \qfr}$, $\tri{C_1, \qfr}$ are overlapping, our statement holds.
\end{proof}

The following result will also prove to be useful.

\begin{lemma}\label{MyLemmaForCn}
	Let $\al, \beta, \gamma \in \Delta_\mathfrak{q}$ be positive roots such that $\al+\beta$, $\beta+\gamma$, $\al+\gamma$, $\al+\beta+\gamma \in \Delta^+_\qfr$, $\gamma-\beta \in \Delta_\mathfrak{h}$, and let $\mathcal{T} := \{\al, \beta, \gamma, \al+\beta, \beta+\gamma, \al+\beta+\gamma\}$. Then $\mathcal{T}$ is K\"ahler or Killing.
\end{lemma}

\begin{proof}
	Since the metric is BTP, it holds that $\nabla^B_{E_\al}T^B(E_\beta, E_\gamma) = 0$. From this, we obtain the following equation:
	$$0 = \Lambda^B(E_\al)T^B(E_\beta, E_\gamma) - T^B(\Lambda^B(E_\al)E_\beta, E_\gamma) - T^B(E_\beta,\Lambda^B(E_\al)E_\gamma).$$
	Since $\gamma - \beta = (\al+\beta) - (\al+\gamma) \in \Delta_\mathfrak{h}$, we deduce that $x_\beta = x_\gamma$ and $x_{\al+\beta} = x_{\al+\gamma}$. From Theorem \ref{EqNomizuTorsion}, we obtain that
	\begin{equation*}
		c_1(x_\beta + x_\gamma - x_{\beta+\gamma})(x_{\al+\beta+\gamma}-x_{\al}) = c_2(x_{\al+\beta}-x_{\al})(x_{\al+\beta}+x_\gamma - x_{\al+\beta+\gamma}).
	\end{equation*}
	where $c_1 := \tfrac{N_{\beta,\gamma}N_{\al,\beta+\gamma}}{x_{\beta+\gamma}}$, $c_2 := \tfrac{N_{\al,\beta}N_{\al+\beta,\gamma} + N_{\al,\gamma}N_{\beta,\al+\gamma}}{x_{\al+\beta}}$. By applying the Jacobi identity to $E_\al, E_\beta, E_\gamma$, we deduce that $N_{\beta,\gamma}N_{\al,\beta+\gamma} = N_{\al,\beta}N_{\al+\beta,\gamma} + N_{\al,\gamma}N_{\beta,\al+\gamma}$. Hence
	\begin{equation}
		\dfrac{(x_\beta + x_\gamma - x_{\beta+\gamma})(x_{\al+\beta+\gamma}-x_{\al})}{x_{\beta+\gamma}} = \dfrac{(x_{\al+\beta}-x_{\al})(x_{\al+\beta}+x_\gamma-x_{\al+\beta+\gamma})}{x_{\al+\beta}}.
	\end{equation}
	From here, the proof follows in a manner similar to that of Lemma \ref{MYLEMMA}.
\end{proof}

We will now give several characterization results for the BTP flag manifolds corresponding to the classical simple compact Lie groups.

\begin{theorem}\label{FlagAn}
    For $G = SU(n+1)$, every $G$-invariant BTP metric on the complex flag manifolds $(G/H, J)$ is either K\"ahler or a multiple of the standard metric.
\end{theorem}

\begin{proof}
    The result is an immediate consequence of the previous lemma, by noticing that the basis $\Pi = \{ \al_1, \dots, \al_n \} \seq \Delta_\qfr^+$ for the root system of $A_n$ type is $A_n$-like.
\end{proof}

\begin{theorem}\label{FlagBn}
    For $G = SO(2n+1)$, every $G$-invariant BTP metric on the complex flag manifolds $(G/H, J)$ is either K\"ahler or a multiple of the standard metric.
\end{theorem}

\begin{proof}
    We will prove that $\Delta_\qfr^+$ is K\"ahler or Killing, where $\Delta$ is the root system for the Lie algebra $\ggo = \mathfrak{so}(2n+1)$ (of $B_n$ type). We proceed by induction on $n \geq 2$.
    
    For $n = 2$, there is some basis $\Pi = \{\al, \beta \} \seq \Delta$ such that $\Delta^+ = \{ \al, \beta, \al+\beta, \al+2\beta \}$. Define $C := \{ \beta, \al, \beta \}$, which is an $A_3$-like set, and observe that $\tri{C, \qfr}= \Delta_\qfr^+$. Hence by Lemma \ref{MYLEMMA}, $\Delta_\qfr^+$ is either a K\"ahler set or a Killing set.
    
    Suppose that we have proven that $\Delta_\qfr^+$ is either K\"ahler or Killing, for any choice of complementary roots on $\Delta$ the root system of $B_{n-1}$ type for some $n \geq 3$. Consider a basis $\Pi = \{ \al_1, \dots, \al_n \} \seq \Delta^+$ for $\Delta$ as in Figure \ref{FigureBn}, where $\Delta$ is of $B_n$ type. Observe that $C_1 := \Pi$ is $A_n$-like and that $\Pi' := \{ \al_2, \dots, \al_n \}$ is a basis for a root system $\Delta' \seq \Delta$ of $B_{n-1}$ type. Set $(\Delta')^+ := \Delta' \cap \Delta^+$ and $(\Delta')_\qfr^+ := \Delta' \cap \Delta_\qfr^+$. By the induction hypothesis and Lemma \ref{MYLEMMA}, we know that $(\Delta')_\qfr^+$ and $\tri{C_1, \qfr}$ are either K\"ahler or Killing.

	Suppose that $\Pi' \cap \Delta_\qfr$ has at least two elements. Then $(\Delta')_{\qfr}^+$ and $\tri{C_1, \qfr}$ are overlapping, and by Lemma \ref{MyLemmaOverlap}, both are K\"ahler or both are Killing. Let $\beta := \sum\limits_{i = 1}^{n} \al_i \in \Delta_\qfr^+$ be the maximal root in $\tri{C_1}$. Then $\tri{C_1} \cup (\Delta')^+ = \Delta^+ - \{ \beta+ \sum\limits_{s = i}^{n}\al_s  \colon  2 \leq i \leq n\}$. Define $C_2 := \{ \beta, \al_n, \dots, \al_2 \}$, which is $A_n$-like, such that $\tri{C_2, \qfr}$ overlaps with $(\Delta')_{\qfr}^+$. We conclude that $\Delta_\qfr^+ = \tri{C_1, \qfr} \cup (\Delta')_\qfr^+ \cup \tri{C_2, \qfr}$ is either K\"ahler or Killing.
	
	If $\Pi' \cap \Delta_\mathfrak{q} = \{ \al_i \}$ for some $2 \leq i \leq n$, and $\al_1$ is a complementary root, then by Theorem \ref{FormulaBn} it holds that $r(G/H) = 4$ if $i = 2$, and $r(G/H) = 5$ otherwise. Define $\beta_i := \sum\limits_{j = i}^{n} \al_j$. Consider $C : = \{ \beta_2, \al_1, \beta_2 \}$ when $i = 2$, and $C := \{ \al_1, \beta_2, \beta_i \}$ otherwise. In both cases, $\tri{C} = \tri{C, \qfr}$ is either K\"ahler or Killing and it contains a representative of each equivalence class. Hence $\Delta_\qfr^+$ is either K\"ahler or Killing. 
	
	If $\Pi' \cap \Delta_\mathfrak{q} = \emptyset$, then the result follows from Theorem \ref{BTP2summands} since $r(G/H) \leq 2$.
\end{proof}

\begin{theorem}\label{FlagCn}
    For $G = Sp(n)$, every $G$-invariant BTP metric on the complex flag manifolds $(G/H, J)$ is either K\"ahler or a multiple of the standard metric.
\end{theorem}

\begin{proof}
    Let $\Delta_\qfr^+$ be the set of positive complementary roots (with respect to the integrable complex structure $J$) and $\Pi_\qfr := \{ \al_{i_1}, \dots, \al_{i_b}\}$ the simple complementary roots, with $1 \leq i_1 < \cdots < i_b \leq n$ (see Figure \ref{FigureCn}).
    Define $$\beta_1 := \sum\limits_{s = 1}^{i_1}\al_s, \hspace{2mm} \beta_j := \sum\limits_{s = i_{j-1}+1}^{i_j}\al_s, \hspace{1mm} \text{for } 2\leq j \leq b-2.$$
  	Also, define $\beta_{b-1} := \sum\limits_{s = i_{b-2}+1}^{n-1}\al_s$ if $\al_n$ is a complementary root, and $\beta_{b-1}:= \sum\limits_{s = i_{b-2}+1}^{i_{b-1}}\al_s$, $\beta_b := \sum\limits_{s = i_{b-1}+1}^{n-1} \al_s$ if not. Consider
  	$$\Pi' := \begin{cases}
  		\{ \beta_1, \dots, \beta_{b-1}, \al_n \}, &\text{if } \al_n \in \Pi_\qfr, \\
  		\{\beta_1, \dots, \beta_b, \al_n\}, &\text{if } \al_n \in \Pi_\hfr.
  	\end{cases}$$
    In either case, $\Pi'$ is a basis for a root system $\Delta' \seq \Delta$ (of $C_b$ type if $\al_n$ is complementary, and of $C_{b+1}$ type otherwise). Notice that $(\Delta')^+ := \Delta' \cap \Delta^+$ contains a representative of each equivalence class in $\Delta_\qfr^+$ under $\sim$. Then without loss of generality, we may assume that $\al_1, \dots, \al_{n-1}$ are all complementary roots.
        
    We will proceed as in Theorem \ref{FlagBn} and prove the result by induction on $n \geq 3$. For $n = 3$, let $\Pi = \{ \al, \beta, \gamma \}$ be the usual ordered basis for $\Delta$. If $\gamma \in \Pi_\qfr$, i.e., $\Delta = \Delta_\qfr$, then define
   	$C_1 := \Pi$ and $\Pi' := \{\beta, \gamma\}$. Observe that $C_1$ is an $A_3$-like set and that $\Pi'$ is a basis for a root system $\Delta' \seq \Delta$ of $B_2$ type. Additionally, $\tri{C_1}$ and $(\Delta')^+ := \Delta' \cap \Delta^+$ are overlapping. Hence by Theorem \ref{FlagBn}, Lemma \ref{MYLEMMA} and Lemma \ref{MyLemmaOverlap}, we deduce that both are K\"ahler or both are Killing. The positive roots not contained in $\tri{C_1} \cup (\Delta')^+$ are $\al+2\beta+\gamma$ and $2\al+2\beta+\gamma$. Set $C_2 := \{ \al+\beta+\gamma, \beta, \al  \}$, which is $A_3$-like. Since $\tri{C_1}$ and $\tri{C_2}$ are overlapping, we conclude $\Delta_\qfr^+ = \Delta^+$ is either K\"ahler or Killing.
   	
    If $\gamma \in \Pi_\hfr$, then consider $C_1 := \{ \al, \beta, \beta+\gamma \}$. By applying Lemma \ref{MyLemmaForCn}, we obtain that $\tri{1} := \{ \al, \beta, \beta+\gamma, \al+\beta, 2\beta+\gamma,\al+2\beta+\gamma \}$ is either K\"ahler or Killing. Set $C_2 := \{ \al+\beta+\gamma, \beta, \al \}$ and notice that $\tri{1}, \tri{C_2}$ are overlapping. Hence $\Delta_\qfr^+$ is either K\"ahler or Killing.
    
    Suppose we have proven $\Delta_\qfr^+$ is either K\"ahler or Killing, for any choice of complementary roots on $\Delta$ the root system of $C_{n-1}$ type for some $n \geq 4$. Let $\Pi = \{ \al_1, \al_2, \dots, \al_n \}$ be a basis for the root system of $C_n$ type. Define $C_1 := \Pi$ and $\Pi' := \{ \al_2, \dots, \al_n\}$. Then $C_1$ is $A_n$-like and $\Pi'$ is a basis for a root system $\Delta' \seq \Delta$ of $C_{n-1}$ type. Hence by our induction hypothesis and Lemma \ref{MYLEMMA}, we know that $\tri{C_1, \qfr}$ and $(\Delta')_\qfr^+ := \Delta' \cap \Delta_\qfr^+$ are either K\"ahler or Killing. Also, since $n \geq 4$, it holds that $\tri{C_1, \qfr}$ and $(\Delta')_\qfr^+$ are overlapping. By Lemma \ref{MyLemmaOverlap}, both are K\"ahler or both are Killing.
    
    Observe that $$\tri{C_1} \cup (\Delta')^+ = \Delta^+ - \left\{ \beta + \sum\limits_{s = k}^{n-1}\al_s \colon 1\leq k \leq n-1 \right\},$$
    where $\beta := \sum\limits_{s = 1}^{n}\al_s \in \Delta_\qfr^+$. Set $C_2 := \{ \beta, \al_{n-1}, \al_{n-2}, \dots, \al_1\}$. Then $C_2$ is $A_n$-like, $\tri{C_2, \qfr}$ is either K\"ahler or Killing, and it overlaps with $\tri{C_1, \qfr}$. We conclude that $\Delta_\qfr^+ = \tri{C_1, \qfr} \cup (\Delta')_\qfr^+ \cup \tri{C_2, \qfr}$ is either K\"ahler or Killing.
\end{proof}   

\begin{theorem}\label{FlagDn}
    For $G = SO(2n)$, every $G$-invariant BTP metric on the complex flag manifolds $(G/H, J)$ is either K\"ahler or a multiple of the standard metric.
\end{theorem}

\begin{proof}
    We will prove that $\Delta_\qfr^+$ is either K\"ahler or Killing, where $\Delta$ is the root system for the Lie algebra $\ggo = \mathfrak{so}(2n)$ (of $D_n$ type). We proceed by induction on $n \geq 4$.

	For $n = 4$, take $\Pi := \{ \al, \beta, \gamma, \delta\} \seq \Delta^+$ the usual ordered basis for $\Delta$ as in Figure \ref{FigureD4} and let $\Pi_\qfr \seq \Pi$ be the simple complementary roots. We may assume that $|\Pi_\qfr | \geq 2$; otherwise, the proof follows from Theorem \ref{BTP2summands} since $r(G/H) \leq 2$.
	
	If $\Pi_\qfr = \Pi$, then take $C_1 := \{ \al, \beta, \gamma \}$, $C_2 := \{ \al, \beta, \delta\}$ and $C_3 := \{ \gamma, \beta, \delta\}$. Each set is $A_3$-like, and the triangles $\tri{C_1}$, $\tri{C_2}$, $\tri{C_3}$ are pairwise overlapping. The leftout roots are $\al+\beta+\gamma+\delta$ and $\al+2\beta+\gamma+\delta$. We define $C_4 := \{ \al+\beta+ \gamma, \delta, \beta\}$, which is $A_3$-like, and whose triangle $\tri{C_4}$ overlaps with $\tri{C_2}$. Hence $\Delta_\qfr^+ = \tri{C_1} \cup \tri{C_2} \cup \tri{C_3} \cup \tri{C_4}$ is either K\"ahler or Killing.
	
	For $|\Pi_\qfr| = 3$, it is enough to consider the cases when $\Pi_\qfr = \{ \al, \beta, \gamma \}$ and $\Pi_\qfr = \{ \al, \gamma, \delta\}$ (every other choice of complementary roots produces equivalent painted Dynkin diagrams). If $\Pi_\qfr = \{ \al, \beta, \gamma \}$, then consider the sets $C_1 := \Pi_\qfr$ and $C_2 := \{\beta, \al, \beta+\gamma \}$, which are both $A_3$-like. Their triangles are overlapping and their union contains a representative of each equivalence class in $\Delta_\qfr^+$. We conclude that $\Delta_\qfr^+$ is either K\"ahler or Killing.
	
	If $\Pi = \{ \al, \gamma, \delta\}$, then consider $C_1 := \{ \al, \beta+\gamma, \delta\}$ and $C_2 := \{\al, \beta+\delta, \gamma\}$ (note that $\beta+ \gamma \sim \gamma$ and $\beta+\delta \sim \delta$). The triangles $\tri{C_1, \qfr}$ and $\tri{C_2, \qfr}$ are both K\"ahler or both Killing, and their union contain representatives of each equivalence class. Hence $\Delta_\qfr^+$ is either K\"ahler or Killing.
	
	For $|\Pi_\qfr| = 2$, it is enough to consider the cases when $\Pi_\qfr = \{ \al, \gamma \}$ and $\Pi_\qfr = \{\al, \beta \}$. The result follows by considering $C_1 := \{\al, \beta+\gamma \}$ if $\Pi_\qfr = \{ \al, \gamma \}$, and $C_2 := \{\beta, \al, \beta+\gamma+\delta \}$ if $\Pi_\qfr = \{\al, \beta \}$. In each scenario, $C_i$ is an $A_2$-like or an $A_3$-like set, whose triangle contains a representative of each equivalence class. We conclude that $\Delta_\qfr^+$ is either K\"ahler or Killing.

	\begin{figure}[h]
		\centering
	\begin{tikzpicture}[scale=1, baseline={(0,-0.1)}]
		\tikzstyle{every node}=[circle, draw, inner sep=1.5pt];
		
		\node (a1) at (-1,0) {};
		\node (a2) at (0,0) {};
		\node (a3) at (0.5,0.866) {};
		\node (a4) at (0.5,-0.866) {};
		
		\node[draw = none] at ($(a1.south)+(0,-8pt)$) {$\al$};
		\node[draw = none] at ($(a2.south)+(8pt, 0)$) {$\beta$};
		\node[below=2pt of dots, draw=none] {};
		\node[draw = none] at ($(a3.south)+(8pt,0)$) {$\gamma$};
		\node[draw = none] at ($(a4.south)+(8pt,0)$) {$\delta$};
		
		\draw (a1)--(a2) (a2)--(a3) (a2)--(a4);
	\end{tikzpicture}
\caption{Dynkin diagram of $D_4$ type}\label{FigureD4}
\end{figure}
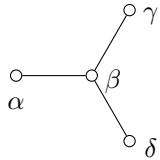
Suppose we have proven $\Delta_\qfr^+$ is either K\"ahler or Killing, for any choice of complementary roots on $\Delta$ the root system of $D_{n-1}$ type for some $n \geq 5$. Let $\Pi = \{ \al_1, \al_2, \dots, \al_n \}$ be a basis for the root system of $D_n$ type as in Figure \ref{FigureDn}. Consider $C_1 := \{ \al_1, \al_2, \dots, \al_{n-1}\}$, $C_2 := \{ \al_1, \al_2, \dots, \al_{n-2}, \al_n\}$ and $\Pi' := \{ \al_2, \al_3, \dots, \al_{n-1}, \al_n\}$, and observe that $C_1$, $C_2$ are $A_{n-1}$-like sets and $\Pi'$ is a basis for a root system $\Delta'\seq \Delta$ of $D_{n-1}$ type. Note that $\tri{C_1, \qfr}$, $\tri{C_2, \qfr}$ and $(\Delta')_\qfr^+ := \Delta' \cap \Delta_\qfr^+$ are either K\"ahler sets or Killing sets.

If at least two of the three intersections $C_1 \cap C_2$, $C_1\cap \Pi'$, $C_2 \cap \Pi'$ each contain at least two complementary roots, then $\tri{C_1, \qfr}$, $\tri{C_2, \qfr}$,  $(\Delta')_\qfr^+$ are all K\"ahler or all Killing (since we can find at least two overlapping pairs). Observe that 
$$\tri{C_1} \cup \tri{C_2} \cup (\Delta')^+ = \Delta^+ - \left\{ \beta + \sum\limits_{s = i}^{n-1}\al_s \colon 2 \leq i \leq n-1 \right\},$$
where $\beta := \sum\limits_{k = 1}^{n-2} \al_k + \al_n \in \tri{C_2}$. Define $C_3 := \{ \beta, \al_{n-1}, \dots, \al_2\}$ if $|C_1 \cap \Pi'\cap \Pi_\qfr| \geq 2$, and $C_3 := \{ \gamma, \al_n, \al_{n-2}, \dots, \al_2\}$ if $|C_2 \cap \Pi' \cap \Pi_\qfr| \geq 2$, where $\gamma := \sum\limits_{k = 1}^{n-1} \al_k \in \tri{C_1}$. In either case, $\tri{C_3, \qfr}$ overlaps with $(\Delta')_\qfr^+$. Hence $\Delta_\qfr^+ = \tri{C_1, \qfr} \cup \tri{C_2, \qfr} \cup (\Delta')_\qfr^+ \cup \tri{C_3, \qfr}$ is either K\"ahler or Killing.

If at most one of the three intersections $C_1 \cap C_2$, $C_1 \cap \Pi'$, $C_2 \cap \Pi'$ contains more than one complementary root, then we either have that $\Pi_\qfr \seq \{ \al_1, \al_{n-1}, \al_n \}$ or $\Pi_\qfr = \{ \al_1, \al_i \}$, $\{ \al_{n-1}, \al_i \}$, $\{ \al_n, \al_i \}$, $\{ \al_i\}$ for some $2 \leq i \leq n-2$. For the first scenario, consider $\beta := \sum\limits_{s =2}^{n-2}\al_s$ and $\Pi'' := \{ \al_1, \beta, \al_{n-1}, \al_n\}$. Then $\Pi''$ is a basis for a root system $\Delta'' \seq \Delta$ of $D_4$ type. Note that $(\Delta'')_\qfr^+:= \Delta'' \cap \Delta_\qfr^+$ contains a representative of each equivalence class of $\Delta_\qfr^+$ under $\sim$. It follows that $(\Delta'')_\qfr^+$, and therefore $\Delta_\qfr^+$, is either a K\"ahler set or a Killing set.

For $\Pi_\qfr = \{\al_1, \al_i \}$, consider the set $C := \left\{\sum\limits_{s = 2}^{n-1}\al_s, \al_1, \sum\limits_{s = 2}^{n-2}\al_s + \al_n\right\}$ if $i = 2$, and $C := \left\{ \al_1, \sum\limits_{s = 2}^{n-2}\al_s, \sum\limits_{s = i}^{n}\al_s\right\}$ if $i > 2$. In either case, $C$ is a $A_3$-like and $\tri{C, \qfr}$ contains a representative of each equivalence class in $\Delta_\qfr^+$. Hence $\Delta_\qfr^+$ is either a K\"ahler set or a Killing set.

For $\Pi_\qfr = \{ \al_i, \al_{n-1}\}$, consider the $A_3$-like set $C := \left\{ \sum\limits_{s = i}^{n-2}\al_s, \al_{n-1}, \sum\limits_{s = i}^{n-2}\al_s +\al_n \right\}$. Its triangle $\tri{C, \qfr}$ contains a representative of each equivalence class in $\Delta_\qfr^+$. Therefore $\Delta_\qfr^+$ is either K\"ahler or Killing. The case $\Pi_\qfr = \{\al_i, \al_n \}$ follows in much the same way.

The proof when $\Pi_\qfr = \{\al_i \}$ follows from Theorem $\ref{BTP2summands}$ since $r(G/H) = 2$.

\end{proof}

\begin{theorem}\label{FlagEn}
	For $n = 6,7$ and $8$, every $E_n$-invariant BTP metric on the complex full flag manifolds $(E_n/T^n, J)$ is either K\"ahler or a multiple of the standard metric.
\end{theorem}
\begin{proof}
	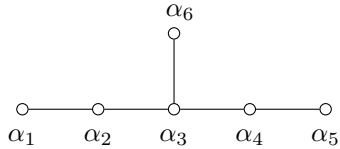
\begin{figure}[h]\label{DynkinE6}
		\centering
		\begin{tikzpicture}[scale = 1, baseline={(0,0)}]
			\tikzstyle{every node}=[circle, draw, inner sep=1.5pt];
			
			\node (a1) at (0,0) {};
			\node (a2) at (1,0) {};
			\node (a3) at (2,0) {};
			\node (a4) at (3,0) {};
			\node (a5) at (4,0) {};
			\node (a6) at (2,1) {};
			
			\node[draw = none] at ($(a1.south)+(0,-8pt)$) {$\alpha_1$};
			\node[draw = none] at ($(a2.south)+(0,-8pt)$) {$\alpha_2$};
			\node[below=2pt of dots, draw=none] {};
			\node[draw = none] at ($(a3.south)+(0,-8pt)$) {$\alpha_3$};
			\node[draw = none] at ($(a4.south)+(0,-8pt)$) {$\alpha_4$};
			\node[draw = none] at ($(a5.south)+(0,-8pt)$) {$\alpha_5$};
			\node[draw = none] at ($(a6.east)+(0,8pt)$) {$\alpha_6$};
			
			\draw (a1)--(a2)--(a3)--(a4)--(a5) (a3)--(a6);
		\end{tikzpicture}
		\caption{Dynkin Diagram of $E_6$ type}
	\end{figure}
	For $n = 6$, let $\Pi := \{ \alpha_1, \dots, \alpha_6\}$ be a basis for $\Delta$ as in Figure \ref{DynkinE6}. Let $\Pi_1 := \Pi - \{ \alpha_5\}$ and $\Pi_2 := \Pi - \{ \alpha_1\}$. Those are basis for root systems $\Delta_1$, $\Delta_2 \subseteq \Delta$ both of $D_5$ type. Additionally $\Delta_1^+ := \Delta_1 \cap \Delta^+$ and $\Delta_2^+ := \Delta_2 \cap \Delta^+$ are overlapping. Therefore, they are either both Kähler or both Killing. Define
	\begin{align*}
		\beta &:= \alpha_1+\alpha_2+\alpha_3+\alpha_4+\alpha_5,& C_1 &:= \{ \alpha_1+\alpha_2+\alpha_3, \alpha_4, \alpha_5 \},\\
		\gamma &:= \alpha_2+2\alpha_3+2\alpha_4+\alpha_5+\alpha_6,& C_2 &:= \{ \beta, \alpha_6, \alpha_3, \alpha_2 \},\\
		& & C_3 &:= \{ \gamma, \alpha_1, \alpha_2, \alpha_3, \alpha_6\}.
	\end{align*}
	The sets $C_1$, $C_2$ and $C_3$ are $A_3$, $A_4$ and $A_5$-like, respectively, and their triangles overlap either with $\Delta_1^+$ or with $\Delta_2^+$. Since it holds that $\Delta^+ = \mathcal{T}_{C_1} \cup \mathcal{T}_{C_2} \cup \mathcal{T}_{C_3} \cup \Delta_1^+ \cup \Delta_2^+$,
	we can conclude that $\Delta^+$ is either Kähler or Killing.
	
	\begin{figure}[htbp]
	\centering
	\begin{tikzpicture}[scale = 1, baseline={(0,0)}]
		\tikzstyle{every node}=[circle, draw, inner sep=1.5pt];
		
		\node (a1) at (0,0) {};
		\node (a2) at (1,0) {};
		\node (a3) at (2,0) {};
		\node (a4) at (3,0) {};
		\node (a5) at (4,0) {};
		\node (a6) at (5,0) {};
		\node (a7) at (3,1) {};
		
		\node[draw = none] at ($(a1.south)+(0,-8pt)$) {$\alpha_1$};
		\node[draw = none] at ($(a2.south)+(0,-8pt)$) {$\alpha_2$};
		\node[below=2pt of dots, draw=none] {};
		\node[draw = none] at ($(a3.south)+(0,-8pt)$) {$\alpha_3$};
		\node[draw = none] at ($(a4.south)+(0,-8pt)$) {$\alpha_4$};
		\node[draw = none] at ($(a5.south)+(0,-8pt)$) {$\alpha_5$};
		\node[draw = none] at ($(a6.south)+(0,-8pt)$) {$\alpha_6$};
		\node[draw = none] at ($(a7.south)+(0,8pt)$) {$\alpha_7$};
		
		\draw (a1)--(a2)--(a3)--(a4)--(a5)--(a6) (a4)--(a7);
	\end{tikzpicture}
	\caption{Dynkin Diagram of $E_7$ type}
	\label{DynkinE7}
	\end{figure}
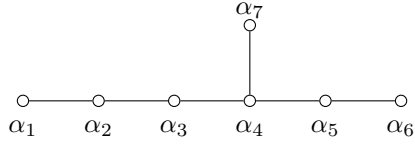
	
	For $n = 7$, let $\Pi := \{ \alpha_1, \dots, \alpha_7 \}$ be a basis for $\Delta$ as in Figure \ref{DynkinE7}. Let $\Pi_1 := \Pi - \{ \alpha_1 \}$ and \linebreak $\Pi_2 := \Pi - \{ \alpha_6\}$ be basis for root systems $\Delta_1$, $\Delta_2 \subseteq \Delta$ of $E_6$ and $D_6$ type, respectively. Note that the sets $\Delta_1^+ := \Delta_1 \cap \Delta^+$ and $\Delta_2^+ := \Delta_2 \cap \Delta^+$ are overlapping. Therefore they are either both Kähler or both Killing.
	Consider the following roots in $\Delta^+$
	\begin{align*}
		\beta_1 &:= \alpha_3 + \alpha_4+\alpha_5+\alpha_6, & \gamma_1 &:= \alpha_3 + 2\alpha_4+2\alpha_5+\alpha_6+\alpha_7,\\
		\beta_2 &:= \beta_1 + \alpha_1+\alpha_2, & \gamma_2 &:= \gamma_1 + \alpha_1+\alpha_2,\\
		\beta_3 &:= \beta_2 + \alpha_2+\alpha_3+\alpha_4+\alpha_7, & \delta &:= \alpha_2+2\alpha_3+3\alpha_4+2\alpha_5+\alpha_6+2\alpha_7.
	\end{align*}
	and the following subsets of $\Delta^+$
	\begin{align*}
		C_1 &:= \{ \beta_1, \alpha_2, \alpha_1 \}, & C_4 &:= \{ \gamma_1, \alpha_2,\alpha_1 \},\\
		C_2 &:= \{ \beta_2, \alpha_7, \alpha_4, \alpha_3, \alpha_2 \}, &  C_5 &:= \{ \gamma_2, \alpha_3, \alpha_4\},\\
		C_3 &:= \{ \beta_3 , \alpha_5, \alpha_4, \alpha_3 \}, & C_6 &:= \{ \delta, \alpha_1, \alpha_2,\alpha_3,\alpha_4,\alpha_5,\alpha_6 \}.
	\end{align*}
	For $1 \leq i \leq 6$ the triangle $\mathcal{T}_i$ overlaps with either $\Delta_1^+$ or with $\Delta_2^+$. Since it holds that the union of them is $\Delta^+$, we conclude that $\Delta^+$ is either Kähler or Killing.
	
	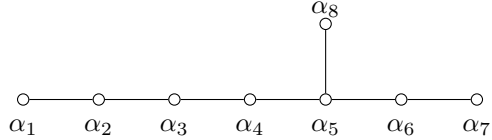
\begin{figure}[htbp]
	\centering
	\begin{tikzpicture}[scale = 1, baseline={(0,0)}]
		\tikzstyle{every node}=[circle, draw, inner sep=1.5pt];
		
		\node (a1) at (0,0) {};
		\node (a2) at (1,0) {};
		\node (a3) at (2,0) {};
		\node (a4) at (3,0) {};
		\node (a5) at (4,0) {};
		\node (a6) at (5,0) {};
		\node (a7) at (6,0) {};
		\node (a8) at (4,1) {};
		
		\node[draw = none] at ($(a1.south)+(0,-8pt)$) {$\alpha_1$};
		\node[draw = none] at ($(a2.south)+(0,-8pt)$) {$\alpha_2$};
		\node[below=2pt of dots, draw=none] {};
		\node[draw = none] at ($(a3.south)+(0,-8pt)$) {$\alpha_3$};
		\node[draw = none] at ($(a4.south)+(0,-8pt)$) {$\alpha_4$};
		\node[draw = none] at ($(a5.south)+(0,-8pt)$) {$\alpha_5$};
		\node[draw = none] at ($(a6.south)+(0,-8pt)$) {$\alpha_6$};
		\node[draw = none] at ($(a7.south)+(0,-8pt)$) {$\alpha_7$};
		\node[draw = none] at ($(a8.south)+(0,8pt)$) {$\alpha_8$};			
		\draw (a1)--(a2)--(a3)--(a4)--(a5)--(a6)--(a7) (a5)--(a8);
	\end{tikzpicture}
	\caption{Dynkin Diagram of $E_8$ type}
	\label{DynkinE8}
	\end{figure}
	
	For $n = 8$, let $\Pi := \{ \alpha_1, \dots, \alpha_8\}$ be a basis for $\Delta$ as in Figure \ref{DynkinE8}. Then the subsets \linebreak $\Pi_1 := \Pi - \{\alpha_7\}$ and $\Pi_2 := \Pi - \{ \alpha_1 \}$ are basis for root systems $\Delta_1$, $\Delta_2 \subseteq \Delta$ of $D_7$ and $E_7$ type, respectively. It holds that $\Delta_1^+ := \Delta_1 \cap \Delta^+$ and $\Delta_2^+ := \Delta_2 \cap \Delta^+$ are overlapping and that $\Delta_i^+$ is either Kähler or Killing for $i = 1,2$. Therefore they are either both Kähler or both Killing. Consider the following elements of $\Delta^+$
	\begin{align*}
		\beta_1 &:= \alpha_3+\alpha_4+\alpha_5+\alpha_6+\alpha_7, & \delta_1 &:= \alpha_2+2\alpha_3+2\alpha_4+3\alpha_5,\\
		\beta_2 &:= \beta_1 + \alpha_1 + \alpha_2, & \delta_2 &:= \delta_1+\alpha_1+\alpha_2,\\
		\beta_3 &:= \beta_2 + \alpha_2+\alpha_3+\alpha_4+\alpha_5+\alpha_8, & \epsilon_1 &:= \beta_1 + \alpha_3+\alpha_4+\alpha_5+\alpha_8,\\
		\beta_4 &:= \beta_3 + \alpha_3+\alpha_4+\alpha_5+\alpha_6, & \epsilon_2 &:= \beta_4 + \alpha_8,\\
		\gamma_1 &:= \alpha_3+\alpha_4+2\alpha_5+2\alpha_6+\alpha_7+\alpha_8, & \epsilon_3 &:= \epsilon_2+\alpha_4+\alpha_5,\\
		\gamma_2 &:= \gamma_1+\alpha_1+\alpha_2, & \eta_1 &:= \epsilon_3 + \alpha_6,\\
		\gamma_3 &:= \gamma_2 + \alpha_4+\alpha_5+\alpha_8, & \epsilon_4 &:= \eta_1+\alpha_5,\\
		& & \eta_2 &:= \epsilon_4 + \alpha_8,
	\end{align*}

	and the following subsets of $\Delta^+$
	\begin{align*}
		C_1 &:= \{\beta_1, \alpha_2, \alpha_1\}, & C_8 &:= \{ \delta_1, \alpha_1, \alpha_2\},\\
		C_2 &:= \{ \beta_2, \alpha_8, \alpha_5, \alpha_4, \alpha_3, \alpha_2 \}, & C_9 &:= \{ \delta_2, \alpha_4,\alpha_5,\alpha_6,\alpha_7\},\\
		C_3 &:= \{ \beta_3, \alpha_6,\alpha_5,\alpha_4,\alpha_3\}, & C_{10} &:= \{ \epsilon_1, \alpha_6,\alpha_5,\alpha_4 \},\\
		C_4 &:= \{ \beta_4, \alpha_8,\alpha_5,\alpha_6,\alpha_7 \}, & C_{11} &:= \{ \epsilon_2, \alpha_5, \alpha_4\},\\
		C_5 &:= \{\gamma_1, \alpha_2, \alpha_1\}, & C_{12} &:= \{\epsilon_3, \alpha_6, \alpha_7 \},\\
		C_6 &:= \{\gamma_2, \alpha_4,\alpha_5,\alpha_8\}, & 	C_{13} &:= \{ \eta_1, \alpha_5, \alpha_8\},\\
		C_7 &:= \{ \gamma_3, \alpha_3,\alpha_4,\alpha_5,\alpha_6,\alpha_7\},&C_{14} &:= \{ \epsilon_4, \alpha_7, \alpha_6 \},\\
		& & C_{15} &:= \{ \eta_2, \alpha_7,\alpha_6,\alpha_5,\alpha_4,\alpha_3, \alpha_2, \alpha_1\}.
	\end{align*}
	
	For $i \in \{1, \dots, 15\}$, the triangle $\mathcal{T}_i$ overlaps with either $\Delta_1^+$ or with $\Delta_2^+$. Since the union of all of them equals $\Delta^+$, we conclude that $\Delta^+$ is either Kähler or Killing.
	
\end{proof}

\begin{theorem}
	Every $G_2$-invariant BTP metric on the complex flag manifolds $(G_2/H, J)$ is either K\"ahler or a multiple of the standard metric.
\end{theorem}

\begin{proof}
	Let $\Pi = \{ \al, \beta\}$ be a basis for the root system $\Delta$ of $G_2$ type such that 
	$$\Delta^+ = \{ \al, \beta, \al+\beta, 2\al+\beta, 3\al+\beta, 3\al+2\beta\}.$$
	
	If $\Pi_\qfr = \Pi$, then by applying Lemma \ref{MyLemmaSeries} on $\al, \beta$, we obtain that the set $\sca_{\al, \beta}\cup \{\al \}$ is either K\"ahler or Killing. Define $C := \{2\al+\beta, \al, \beta\}$ which is $A_3$-like, and notice that $\tri{C}$ overlaps with $\mathcal{S}_{\al, \beta} \cup \{ \al\}$. Then $\Delta_\qfr^+ = \tri{C} \cup S_{\al, \beta}$ is either K\"ahler or Killing.
	
	If $\Pi_\qfr = \{\al\}$, observe that $\sca_{\al, \al+\beta}\cup \{\al\}$ is a set which is either K\"ahler or Killing, and that contains a representative of each equivalence class in $\Delta_\qfr^+$. We conclude that $\Delta_\qfr^+$ is either K\"ahler or Killing.
	
	If $\Pi_\qfr = \{ \beta\}$, then the result follows from Theorem \ref{BTP2summands}, since $r(G_2/H) = 2$.
\end{proof}

\begin{theorem}
	Every $F_4$-invariant BTP metric on the complex flag manifolds $(F_4/H, J)$ is either K\"ahler or a multiple of the standard metric.
	
\end{theorem}

\begin{proof}
	Let $\Pi := \{ \al_1, \al_2, \al_3, \al_4 \}$ be a basis for the root system $\Delta$ of $F_4$ type as in Figure \ref{FigureF4}. As in the previous theorems, the proof will follow by showing that $\Delta_\qfr^+$ is either a K\"ahler set or a Killing set. This will be done by providing a list of overlapping subsets so that the union of them all contains a representative of each equivalence class of $\Delta_\qfr^+$ under $\sim$, and such that each subset is known to be either a K\"ahler set or a Killing set, by Lemmas \ref{MYLEMMA}, \ref{MyLemmaForCn} and Theorems \ref{FlagBn}, \ref{FlagCn}. We list out said subsets in the following table. The subsets will be denoted by $\Pi_i$ when they are seen as basis for root systems $\Delta_i \seq \Delta$ (which will always be either of $B_3$ type or $C_3$ type), and denoted by $C_i$ when they are seen as $A_k$-like sets for some $k \in \Nmb$. We leave out the cases $\Pi_\qfr = \{ \al_1\}$ and $\Pi_\qfr = \{\al_4\}$ since they follow from Theorem \ref{BTP2summands}.
	
		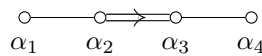
\begin{figure}[H]
		\centering
		\begin{tikzpicture}[scale = 1, baseline={(0,-0.1)}]
			\tikzstyle{every node}=[circle, draw, inner sep=1.5pt];
			
			\node (a3) at (0,0) {};
			\node (a4) at (1,0) {};
			\node (a5) at (2,0) {};
			\node (a6) at (3,0) {};
			
			\node[draw = none] at ($(a3.south)+(0,-8pt)$) {$\al_1$};
			\node[draw = none] at ($(a4.south)+(0,-8pt)$) {$\al_2$};
			\node[draw = none] at ($(a5.south)+(0,-8pt)$) {$\al_3$};
			\node[draw = none] at ($(a6.south)+(0,-8pt)$) {$\al_4$};
			
			\draw (a3)--(a4) (a5)--(a6);
			\draw[double distance=2pt] (a4) -- (a5);
			\draw[double distance=2pt,-implies] (a4)-- ($(a5)!0.4!(a4)$); 
		\end{tikzpicture}
		\caption{Dynkin diagram of $F_4$ type}\label{FigureF4}
	\end{figure}

	\begin{table}[H]
	\centering
	\small
	\setlength{\tabcolsep}{3pt} 
	\renewcommand{\arraystretch}{1.4}
	
	\makeatletter
	\newcommand{\lineagruesa}{\noalign{\global\arrayrulewidth=1.5pt}\hline\noalign{\global\arrayrulewidth=0.4pt}}
	\makeatother
	
	\label{tab:listforF4}
		\vspace{-3mm}
			\begin{tabular}{|c |p{10cm}@{}|}
				\lineagruesa
				Painted Dynkin diagram & List of subsets \\
				\lineagruesa
				\DynF{white}{white}{white}{white} &
				
				$\Pi_1 := \{\al_1, \al_2, \al_3\}$, $\Pi_2 := \{\al_2,\al_3, \al_4 \}$, $C_1 := \{\al_1+\al_2, \al_3, \al_4 \}$, \par
				$C_2 := \{\al_2+2\al_3+\al_4,\al_1,\al_2,\al_3 \}$, $C_3 := \{\al_2+2\al_3+2\al_4,\al_1,\al_2,\al_3 \},$ \par
				$C_4 := \{\al_1+2\al_2+3\al_3+2\al_4, \al_3,\al_2,\al_1\}$\\ \hline
				\DynF{white}{white}{white}{black} &
				$\Pi_1 := \{\al_1, \al_2, \al_3 \}$, $C_1 := \{\al_1+\al_2+2\al_3+2\al_4, \al_2,\al_3 \},$ \par
				$C_2 :=\{\al_1+2\al_2+3\al_3+2\al_4, \al_3,\al_2,\al_1 \}.$\\ \hline
				\DynF{white}{white}{black}{white} & $\Pi_1 := \{\al_1,\al_2,\al_3 \}$, $\Pi_2 := \{\al_2,\al_3,\al_4\}$, $C_1 := \{\al_1,\al_2,\al_3,\al_4\},$ \par
				$C_2 :=\{\al_2+2\al_3+2\al_4,\al_1,\al_2\}$, $C_3 := \{\al_2+2\al_3+\al_4,\al_1,\al_2 \}$, \par
				$C_4 := \{\al_1+2\al_2+4\al_3+2\al_4,\al_2,\al_1\}.$\\ \hline
				\DynF{white}{black}{white}{white} & $\Pi_1 := \{\al_1,\al_2,\al_3\}$, $\Pi_2 := \{\al_2,\al_3,\al_4\}$, $C_1 := \{\al_1,\al_2,\al_3,\al_4 \}$, \par $C_2 := \{\al_1+\al_2+\al_3+\al_4, \al_3, \al_4 \}$, $C_3 := \{\al_1+2\al_2+2\al_3,\al_4,\al_3 \},$ \par
				$C_4 := \{\al_1+2\al_2+2\al_3+2\al_4, \al_3,\al_2+\al_3 \},$ \par
				$C_5 := \{\al_1+2\al_2+3\al_3+2\al_4,\al_2+\al_3,\al_1 \}.$\\ \hline
				\DynF{black}{white}{white}{white} & 
				$\Pi_1 := \{\al_1, \al_2,\al_3\}$,
				$\Pi_2 :=\{\al_2,\al_3,\al_4\}$, $C_1 := \{\al_1+2\al_2+2\al_3, \al_4,\al_3\},$ \par
				$C_2 := \{\al_1+\al_2+2\al_3+2\al_4, \al_2,\al_3\}$, $C_3 := \{\al_1+2\al_2+3\al_3+2\al_4, \al_3,\al_2\}.$ \\ \hline
				\DynF{white}{white}{black}{black} & 
				$\Pi_1 := \{\al_1,\al_2,\al_3\}$, $C_1 := \{\al_1+2\al_2+4\al_3+2\al_4,\al_2,\al_1 \}.$\\ \hline
				\DynF{white}{black}{white}{black} &
				$\Pi_1 := \{\al_1, \al_2, \al_3\}$, $C_1 := \{\al_2+2\al_3+2\al_4, \al_1+\al_2, \al_3 \},$ \par
				$C_2 := \{\al_1+2\al_2+3\al_3+2\al_4, \al_2+\al_3,\al_1\}.$ \\ \hline
				\DynF{black}{white}{white}{black} & 
				$\Pi_1 := \{\al_1,\al_2,\al_3\}$, $C_1 := \{\al_2+2\al_3+ 2\al_4,\al_1+\al_2,\al_3\},$ \par
				$C_2 := \{\al_1+2\al_2+3\al_3+2\al_4,\al_3,\al_2 \}.$ \\ \hline
				\DynF{white}{black}{black}{white} & 
				$C_1 := \{ \al_1,\al_2+\al_3+\al_4,\al_3+\al_4\},$ \par
				$C_2 := \{\al_1+2\al_2+3\al_3+\al_4,\al_2+\al_3+\al_4,\al_1\}.$ \\ \hline
				\DynF{black}{white}{black}{white} & $\Pi_1 := \{\al_2,\al_3,\al_4\}$, $C_1 := \{\al_2,\al_1+\al_2+2\al_3,\al_4 \},$ \par
				$C_2 := \{\al_1+2\al_2+3\al_3+\al_4,\al_3+\al_4,\al_2\}.$ \\ \hline
				\DynF{black}{black}{white}{white} & 
				$\Pi_1 := \{ \al_2,\al_3,\al_4\}$, $C_1 := \{\al_1+2\al_2+2\al_3,\al_4,\al_3\},$ \par 
				$C_2 := \{\al_1+2\al_2+3\al_3+\al_4,\al_4,\al_3\}.$ \\ \hline
				\DynF{black}{white}{black}{black} & 
				$C_1 := \{\al_2,\al_1+\al_2+2\al_3+2\al_4,\al_2+2\al_3 \}.$ \\ \hline
				\DynF{black}{black}{white}{black} &
				$C_1 := \{\al_1+\al_2+\al_3,\al_3+\al_4,\al_2+\al_3+\al_4 \}$\footnotemark,
				\par
				$C_2 := \{\al_3,\al_1+2\al_2+2\al_3+2\al_4,\al_3\}.$ \\ \hline
			\end{tabular}
		\end{table}

	\footnotetext{This set is not $A_3$-like, but it produces a triangle in the sense of Lemma \ref{MyLemmaForCn}.}
	
	In any case, all equivalence classes in $\Delta_\qfr^+$ have nontrivial intersection with the union of the respective complementary roots $(\Delta_i)^+_\qfr : = \Delta_i^+ \cap \Delta_\qfr$ corresponding to $\Pi_i$ and the triangles $\tri{C_i , \qfr}$. Hence we conclude that $\Delta_\qfr^+$ is either a K\"ahler set or a Killing set.
\end{proof}

\newpage
\bibliographystyle{plain}
\bibliography{References}

\end{document}